\documentclass[11pt]{article}
\usepackage{makeidx}
\usepackage{amsmath,amssymb,amsthm}
\usepackage{graphicx}
\usepackage{color}
\usepackage{verbatim}
\usepackage{longtable}
\usepackage{lscape}

\definecolor{blue}{rgb}{0,0,0.9}
\definecolor{red}{rgb}{0.9,0,0}
\definecolor{green}{rgb}{0,0.9,0}

\newtheorem{lemma}{Lemma}

\newtheorem{theorem}{Theorem}

\def\lam{\lambda} \def\alp{\alpha}

\def\inprod#1#2{\langle#1,#2\rangle}
\def\cA{{\cal A}}
\def\cS{{\cal S}}
\newcommand{\ff}{\boldsymbol f}

\newcommand{\f}{\boldsymbol f}
\newcommand{\bP}{\boldsymbol P}
\newcommand{\bF}{\boldsymbol F}
\newcommand{\bv}{\boldsymbol v}
\newcommand{\bz}{\boldsymbol z}
\newcommand{\R}{\mathbb R}
\newcommand{\N}{\mathbb N}
\newcommand{\K}{\mathbf K}
\newcommand{\A}{\mathbf A}
\newcommand{\B}{\mathbf B}
\newcommand{\M}{\mathbf M}
\newcommand{\y}{\mathbf y}
\newcommand{\q}{\mathbf q}
\newcommand{\ms}{\mathcal{S}}
\newcommand{\blambda}{\mathbf \lambda}

\begin{document}

\title{A bounded degree SOS hierarchy for polynomial optimization\thanks{The work of the first author is partially 
supported by a PGMO grant from {\it Fondation Math\'ematique Jacques Hadamard, and 
a grant from the European Research Council (ERC-AdG: TAMING project).}}}

\author{
Jean B. Lasserre\thanks{LAAS-CNRS and Institute of Mathematics, University of
Toulouse, LAAS, 31031 Toulouse c\'{e}dex 4, France ({\tt lasserre@laas.fr}).},
\
 Kim-Chuan Toh\thanks{Department of Mathematics, National
         University of Singapore, 
        10 Lower Kent Ridge Road, Singapore
         119076 ({\tt mattohkc@nus.edu.sg}).
         } 
\ and
Shouguang Yang\thanks{Department of Mathematics, National University of
         Singapore, 10 Lower Kent Ridge Road, Singapore 119076 ({\tt matyash@nus.edu.sg}).         
        }
}
\maketitle
\date

\begin{abstract}
We consider a  new hierarchy of semidefinite relaxations
for the general polynomial optimization problem
$(P):\:f^{\ast}=\min \{\,f(x):x\in K\,\}$
on a compact basic semi-algebraic set $K\subset\R^n$.  
This hierarchy 
combines some advantages of the standard LP-relaxations associated with Krivine's positivity certificate
and some advantages of the standard SOS-hierarchy. In particular it has the following attractive features: (a) In contrast to
the standard SOS-hierarchy, for each relaxation in the hierarchy, the size of the 
matrix associated with the semidefinite constraint is the same and fixed in advance
by the user. (b) In contrast to the LP-hierarchy, finite convergence occurs at the first step of the hierarchy 
for an important class of convex problems. Finally (c) some important techniques 
related to the use of  point evaluations for declaring a polynomial to be zero
and to the use of rank-one matrices make an efficient implementation possible.
Preliminary results on a sample of non convex problems are encouraging.
\end{abstract}

\section{Introduction}

We consider the polynomial optimization problem:
\begin{equation}
\label{def-pb}
(P):\quad f^*=\displaystyle\min_x\:\{f(x)\::\: x\in\K\:\}
\end{equation}
where $f\in\R[x]$ is a polynomial and $\K\subset\R^n$ is the basic semi-algebraic set
\begin{equation}
\label{setk}
\K\,=\,\{\,x\in\R^n\::\: g_j(x)\,\geq\,0,\quad j=1,\ldots,m\},
\end{equation}
for some polynomials $g_j\in\R[x]$, $j=1,\ldots,m$. In order to approximate (and sometimes solve exactly)
$(P)$ one may instead solve a hierarchy of {\it convex relaxations} of $(P)$ of increasing sizes, namely for instance:

$\bullet$ {\it Semidefinite} relaxations based on Putinar's certificate of positivity on $K$ \cite{Putinar93},
where the $d$-th convex relaxation of the hierarchy is a semidefinite program 
{given by}
\begin{equation}
\label{put-cert}
\gamma_d=\max_{t,\sigma_j}\:\left\{\,t\::\: f-t=\sigma_0+\sum_{j=1}^m \sigma_j\,g_j\,\right\}.
\end{equation}
The unknowns $\sigma_j$ are sums of squares (SOS) polynomials with the degree bound constraint, ${\rm degree}(\sigma_jg_j)\leq 2d$, $j=0,\ldots,m$,
and the expression  in (\ref{put-cert}) is a certificate of positivity on $\K$ for the polynomial $x\mapsto f(x)-t$.\\

$\bullet$ {\it LP}-relaxations based on Krivine-Stengle's certificate of positivity on $\K$ \cite{Krivine64a,Stengle74},
where the $d$-th convex relaxation of the hierarchy is a linear program 
{given by}
\begin{equation}
\label{kriv-cert}
\theta_d\,=\,\max_{\lambda\geq0,t}\:\left\{t\::\:f-t\,=\,\sum_{(\alpha,\beta)\in\N^{2m}_{d}}\lambda_{\alpha\beta}\,\prod_{j=1}^m\left(g_j^{\alpha_j}\,(1-g_j)^{\beta_j}\right)\right\},
\end{equation}
where $\N^{2m}_{d}=\{(\alpha,\beta)\in\N^{2m}:\sum_j\alpha_j+\beta_j\leq d\}$. The unknown are $t$ and the nonnegative scalars $\lambda=(\lambda_{\alpha\beta})$, and
it is assumed that $0\leq g_j\leq 1$ on $\K$ (possibly after scaling) and the family $\{1,g_j\}$ generates the algebra $\R[x]$ of polynomials.
Problem (\ref{kriv-cert}) is an LP because stating that the two polynomials in both sides of ``$=$" are equal yields linear constraints 
on the $\lambda_{\alpha\beta}$'s.
For instance, the LP-hierarchy from Sherali-Adams RLT  \cite{SheraliAdams90} and their variants \cite{SheraliAdams99} are of this form.

In both cases, $(\gamma_d)$ and $(\theta_d)$, $d\in\N$, provide two monotone nondecreasing sequences of lower bounds 
on $f^\ast$ and if $\K$ is compact, then both converge to $f^\ast$ as one lets $d$ increases.
For more details as well as a comparison of such relaxations, 
the reader is referred to e.g. 
Lasserre \cite{lasserre-imperial,Lasserre02a} and Laurent \cite{Laurent05}, as well as Chlamtac and Tulsiani \cite{tulsiani} 
for the impact of LP- and SOS-hierarchies on approximation algorithms in combinatorial optimization.

Of course, in principle,  one would much prefer to solve LP-relaxations rather than semidefinite relaxations (i.e. compute
$\theta_d$ rather than $\gamma_d$) because present
LP-software packages can solve sparse 
problems with millions of variables and constraints, which is 
far from being the case for today's semidefinite solvers. And so the hierarchy (\ref{put-cert}) applies to problems of modest size only 
unless some sparsity or symmetry is taken into account in which case specialized variants can handle problems of much larger size; see e.g. Waki et al. \cite{waki}.
However, on the other hand, the LP-relaxations (\ref{kriv-cert}) suffer from several serious 
theoretical and practical drawbacks. For instance, it has been shown in \cite{Lasserre02a,lasserre-imperial} that the LP-relaxations 
{\it cannot} be exact for most convex problems, i.e.,  the sequence of the associated optimal values converges
to the global optimum only {\it asymptotically} and not in finitely many steps. Moreover, the LPs of the hierarchy are numerically ill-conditioned.
This is in contrast with the semidefinite relaxations (\ref{put-cert})
for which finite convergence takes place for convex problems where $\nabla^2f(x^\ast)$ is positive definite at every minimizer $x^\ast\in\K$
(see de~ Klerk and Laurent \cite[Corollary 3.3]{deklerk-convex}) and occurs at the first relaxation for SOS-convex\footnote{An SOS-convex polynomial is a convex polynomial whose Hessian factors as $L(x)L(x)^T$ for some rectangular matrix polynomial $L$.
For instance, separable convex polynomials are SOS-convex.} problems \cite[Theorem 3.3]{lasserre-convex}.
In fact, as demonstrated in recent  works of Marshall \cite{Marshall2} and Nie \cite{nie2012},
finite convergence is generic even for non convex problems.

\subsection{Contribution} 
This paper is in the vein of recent attempts in Lasserre \cite{lasserre-new} and Ahmadi and Majumdar \cite{ahmadi} 
to overcome the important  computational burden associated with the standard SOS-hierarchy (\ref{put-cert}).
In particular, in \cite{lasserre-new} we have suggested another hierarchy
of convex relaxations which combines some of the advantages of the SOS- and LP- hierarchies (\ref{put-cert}) and (\ref{kriv-cert}). 
In the present paper we take advantage of attractive features of the SDPT3 solver \cite{TTT99,TTT03} to provide an effective 
implementation of this new hierarchy. First preliminary tests on a sample
of non convex problems are encouraging and suggest that this new hierarchy might be efficient. This new hierarchy is 
another type of SOS-hierarchy labelled BSOS (for hierarchy with {\it bounded} degree SOS) with the following attractive features:

$\bullet$ In contrast to the standard SOS-hierarchy (\ref{put-cert}), for {\it each} semidefinite program in the hierarchy,
the size ${n+k\choose n}$ of the semidefinite matrix variable is now fixed, parametrized by an integer $k$ that one fixes in advance. This integer 
$k$ determines the degree of a certain SOS polynomial (for instance one may fix $k=2$), whence the label BSOS (for ``bounded"-SOS). 
Recall that in the standard SOS-hierarchy (\ref{put-cert})
the size of the semidefinite matrix variable is ${n+d\choose n}$ with rank $d$ in the hierarchy. 

$\bullet$ In contrast to the LP-hierarchy (\ref{kriv-cert}), {\it finite} convergence occurs at the first step in the hierarchy for a large class of convex 
problems; typically convex problems defined with convex quadratic polynomials or SOS-convex polynomials of degree at most $k$. 
Recall that such finite convergence is impossible for the LP-hierarchy (\ref{kriv-cert}).

$\bullet$ 
{Just as in}
 the standard SOS-hierarchy (\ref{put-cert}), there also exists a sufficient condition for finite convergence of the hierarchy. Namely
it suffices to check whether at an optimal solution of the corresponding SDP, some 
associated moment matrix is rank-one.

$\bullet$ Last but not least, to implement this hierarchy one uses
important techniques that dramatically alleviate the computational burden associated with a standard (careless) implementation.
Namely, (a) to declare that two polynomials are identical one uses that their values are equal on finitely many randomly chosen points
(instead of equating their coefficients), and (b) the SDP solver SDPT3 \cite{TTT99,TTT03} can be used to handle efficiently
some type of matrices used in our positivity certificate.

\paragraph{Preliminary computational experiments}
First we have compared our results with those obtained with the GloptiPoly software \cite{gloptipoly}
(devoted to solving the SOS-hierarchy (\ref{put-cert})) on a sample of non convex problems with up to $20$ variables.
For problems with low degree (in the initial data) and/or low dimension we obtain the global optimum whereas good lower bounds are always obtained for problems with high degee or higher dimension (e.g. problems with degree $4$ and up to $20$ variables). 

Next, we have also tested the LP-hierarchy (\ref{kriv-cert}) on a sample of convex problems and as expected the convergence is very poor and the resulting LPs become ill-conditioned. 
In addition, the LP can be expensive to solve as  the LP data is typically dense.
In contrast, the new hierarchy 
(with smallest value $k=1$ of its parameter) converges at the first step even though some of the problems 
are defined with polynomials of degree larger than $2$.

We have also considered a sample of non convex quadratic problems of the form $\inf \{x^TAx: x\in\Delta\}$ where $\Delta\subset\R^n$ is the canonical simplex and $A$ is a randomly generated real symmetric matrix with $r$ negative eigenvalues and $n-r$ positive eigenvalues. For all problems that could be solved with GloptiPoly (up to $n=20$ variables) we obtain the optimal values. For the other problems with $n=40, 50,100$ variables, only the first (dense) relaxation of GloptiPoly can be implemented
and yields only a lower bound on the global optimum. For those problems, 
a better lower bound is obtained in a reasonable amount of time by running the BSOS hierarchy. 

Finally we have considered the minimization of quadratic and quartic polynomials (with up to $40$ variables)
on the Euclidean unit ball intersected with the positive orthant. Again in those examples
only the first SDP-relaxation of GloptiPoly can be implemented, providing only a lower bound. In contrast BSOS
solves all problems at step $2$ of the hierarchy in a reasonable amount of time.

Of course this new hierarchy of semidefinite relaxations also has  its drawbacks (at least in its present version). Namely some submatrix (of the matrix used to describe the linear equality constraints of the resulting SDP) is fully dense and many of
these linear constraints are nearly dependent, which yields a lack of accuracy in the optimal solution when the order of relaxation $d$ is increased.

\section{Main result}

\subsection{Notation and definitions}
Let $\R[x]$ be the ring of polynomials in the variables
$x=(x_1,\ldots,x_n)$.
Denote by $\R[x]_d\subset\R[x]$ the vector space of
polynomials of degree at most $d$, which forms a vector space of dimension $s(d)={n+d\choose d}$, with e.g.,
the usual canonical basis $(x^\alpha)$ of monomials.
Also, denote by $\Sigma[x]\subset\R[x]$ (resp. $\Sigma[x]_d\subset\R[x]_{2d}$)
the space of sums of squares (SOS) polynomials (resp. SOS polynomials of degree at most $2d$). 
If $f\in\R[x]_d$, we write
$f(x)=\sum_{\alpha\in\N^n_d}f_\alpha x^\alpha$ in the canonical basis and
denote by $\f=(f_\alpha)\in\R^{s(d)}$ its vector of coefficients. 
Finally, let $\ms^n$ denote the space of 
$n\times n$ real symmetric matrices, with inner product $\langle \A,\B\rangle ={\rm trace}\,\A\B$. We use the notation 
$\A\succeq0$ (resp. $\A\succ0$) to denote that
 $\A$ is positive semidefinite (definite).  With $g_0:=1$, the quadratic module $Q(g_1,\ldots,g_m)\subset\R[x]$ generated by
polynomials $g_1,\ldots,g_m$, is defined by
\[Q(g_1,\ldots,g_m)\,:=\,\left\{\sum_{j=0}^m\sigma_j\,g_j\::\:\sigma_j\in\Sigma[x]\,\right\}.\]
We briefly recall two important theorems by Putinar \cite{Putinar93} and Krivine-Stengle \cite{Krivine64a,Stengle74} respectively,
on the representation of polynomials that are positive on $\K$.
\begin{theorem}
\label{th-prelim}
Let $g_0=1$ and  $\K$ in (\ref{setk}) be compact.

(a) If the quadratic polynomial $x\mapsto M-\Vert x\Vert^2$
belongs to $Q(g_1,\ldots,g_m)$ and if $f\in\R[x]$ is strictly positive on $\K$ then $f\in Q(g_1,\ldots,g_m)$.

(b) Assume that $0\leq g_j\leq 1$ on $\K$ for every $j$, and the family $\{1,g_j\}$ generates $\R[x]$. If $f$ is strictly positive on $\K$ then 
\[f\,=\,\sum_{\alpha,\beta\in\N^m}c_{\alpha\beta}\,\prod_{j}\left(g_j^{\alpha_j}\,(1-g_j)^{\beta_j}\right),\]
for some (finitely many) nonnegative scalars $(c_{\alpha\beta})$.
\end{theorem}

\subsection{The Bounded-SOS-hierarchy (BSOS)}

Consider the problem
\begin{eqnarray*}
 {(P):} \quad f^\ast = \min\{ f(x) \mid x\in \K\}
\end{eqnarray*}
where $K\subset\R^n$ is the basic semi-algebraic set defined in (\ref{setk}),
assumed to 
be compact. Moreover we also  assume that $g_j(x) \leq 1$ for all $x\in \K$ and $j=1,\dots,m$,
and $\{1,g_j\}$ generates the ring of polynomials $\R[x]$.  With $k\in\N$ fixed, consider the family  of optimization problems indexed by $d\in\N$:
\begin{equation}
\label{aux_{1}}
q_d^k := \sup_{\blambda,t}\,\{\, t \mid
 L_d(x,\lam) - t \in \Sigma[x]_{k}, \; \lam \geq 0\,\},\qquad d\in\N.
\end{equation}
Observe that when $k$ is fixed, then for each $d\in\N$:
\begin{itemize}
\item Computing $q^k_d$ in (\ref{aux_{1}}) reduces to  solving a semidefinite program and therefore
(\ref{aux_{1}}) defines a {\it hierarchy} of semidefinite programs because $q^k_{d+1}\geq q^k_d$ for all $d\in\N$.
\item The semidefinite constraint is
associated with the constraint  $L_d(x,\lam) - t \in \Sigma[x]_{k}$ and
the associated matrix has fixed size ${n+k\choose n}$, independent of $d\in\N$, a crucial feature for 
computational efficiency of the approach.
\item If $k=0$ then (\ref{aux_{1}}) is the linear program (\ref{kriv-cert}) and so $\theta_d=q^0_d\leq q^k_d$ for all $d,k$.
\end{itemize}

\paragraph{Interpretation.}
For any fixed $d\geq 1$, problem (P) is 
easily seen to be equivalent to the following problem by adding redundant constraints:
\begin{eqnarray*}
 (\tilde{P}): \quad f^\ast = \min\{ f(x) \mid  h_{\alp\beta}(x)\geq 0\;\; \forall \; (\alp,\beta) \in \N_d^{2m}\}
\end{eqnarray*}
where $\N_d^{2m} = \{ (\alp,\beta)\in \N^{2m} \mid |\alp|+|\beta| \leq d\}$
and $h_{\alp\beta}\in\R[x]$ is the polynomial
\[x\,\mapsto\,h_{\alpha\beta}(x)\,:=\,\prod_{j=1}^m g_j(x)^{\alp_j}(1-g_j(x))^{\beta_j},\qquad x\in\R^n.\]
Given $\blambda=(\lambda_{\alpha\beta})$, $(\alpha,\beta)\in\N^{2m}_d$, consider the Lagrangian function:
$$
 x\,\mapsto\,L_d(x,\lam) = f(x) - \sum_{(\alp,\beta)\in\N^{2m}_d}\lam_{\alp\beta} h_{\alp\beta}(x),\qquad x\in\R^n.
$$
The Lagrangian dual of $(\tilde{P})$ is given by 
$$
(\tilde{P}^\ast_d):\quad \sup_{\lambda}\:\{\,G_d(\lambda):\,\lambda\geq0\,\}
$$
where the function $G_d(\cdot)$ is given by:
$$
  \lambda\,\mapsto  G_d(\lambda) := \inf_{x\in\R^n} \{\,L_d(x,\lambda) \},\qquad\lambda\,\geq\,0.
$$

Now for a fixed $\lam$, the evaluation of $G_d(\lam)$ is computational intractable. However, 
let $k\in\N$ be fixed and
observe that
\begin{eqnarray*}
G_d(\lam) = \inf_{x\in \R^n} L_d(x,\lam) &=& \sup_{t}\,\{ \,t \mid L_d(x,\lam)-t \geq 0,\quad\forall\, x\,\}\\
& \geq& 
 \sup_{t}\,\{\, t \mid L_d(x,\lam)-t  \in \Sigma[x]_{k}\,\},
\end{eqnarray*}
where recall that $\Sigma[x]_{k}$ is the space of SOS polynomials of degree at most
$2k$. 
Moreover,
$$
f^\ast\,\geq\, \sup_{\lam \geq 0}\, G_d(\lam) \; \geq\,q^k_d,\qquad \forall\,d\in\N.
$$
So the semidefinite program (\ref{aux_{1}}) can be seen as a tractable simplification of the intractable problem: $\sup_{\lambda\geq0}G(\lambda)$. The linear program (\ref{kriv-cert}) (which is (\ref{aux_{1}}) with $k=0$) is an even more brutal simplification, so that $q^k_d\geq q^0_d=\theta_d$ for all $d,k$.
As a matter of fact we have the more precise and interesting result.

\begin{theorem}[\cite{lasserre-new}]
Let $\K\subset\R^n$ in (\ref{setk}) be compact with nonempty interior and 
$g_j(x) \leq 1$ for $x\in\K$ and $j=1,\dots,m$. Assume further
that the family $\{1,g_j\}$
generates the algebra $\R[x]$. Let $k\in\N$ be fixed and for each $d\in\N$, let $q^k_d$ be as in (\ref{aux_{1}}). Then:

(a) The sequence $(q^k_d)$, $d\in\N$, is monotone nondecreasing and $q^k_d\to f^\ast$ as $d\to\infty$.

(b) Moreover, assume that Slater's condition holds, i.e., there exists $x_0\in K$ such that
$g_j(x_0)>0$ for all $j=1,\ldots,m$. If $f$ and $-g_j$, $j=1,\ldots,m$,  are SOS-convex\footnote{A polynomial $f\in\R[x]$ is SOS-convex if its Hessian $\nabla^2 f$ is an SOS matrix, i.e., $\nabla^2f(x)=L(x)\,L(x)^T$ for some matrix polynomial  $L\in\R[x]^{n\times p}$ and some $p\in\N$.} polynomials of degree at most $2k$ then $q^k_1=f^*$, i.e., 
finite convergence takes places at the first relaxation in the hierarchy! In particular when $f,-g_j$ are convex quadratic polynomials then $q^1_1=f^*$.
\end{theorem}
\begin{proof}
The first result is a direct application of Theorem \ref{th-prelim}(b) since for any integer $d,k$,
$f^*\geq q^k_d\geq q^0_d=\theta_d$, and by Theorem \ref{th-prelim}, $\theta_d\to f^*$ as $d\to\infty$.
Next, if Slater's condition holds and $f$ and $-g_j$ are SOS-convex, $j=1,\ldots,m$,
then there exist nonnegative Lagrange-KKT multipliers $\lambda\in\R^m_+$ such that 
\[\nabla f(x^*)-\sum_j\lambda_j\,\nabla g_j(x^*)\,=\,0;\quad \lambda_j\,g_j(x^*)\,=\,0,\quad j=1,\ldots,m.\]
In other words, the Lagrangian $x\mapsto L(x):=f(x)-f^*-\sum_j\lambda_j\,g_j(x)$ is SOS-convex and satisfies
$L(x^*)=0$ and $\nabla L(x^*)=0$. By Helton and Nie \cite[Lemma 8, p. 33]{helton-nie}, $L$ is SOS (of degree at most $2k$).

\end{proof}
\subsection{The SDP formulation of (\ref{aux_{1}})}

To formulate (\ref{aux_{1}}) as a semidefinite program one has at least two possibilities depending on how we state that 
two polynomials $p,q\in\R[x]_d$ are identical. Either by {\it equating their coefficients} (e.g. in the monomial basis), i.e.,
$p_\alpha=q_\alpha$ for all $\alpha\in\N^n_d$, or by {\it equating their values} on ${n+d\choose n}$ 
generic points (e.g. randomly generated on  the box $[-1,1]^n$). In the present context
of (\ref{aux_{1}})  we prefer the latter option 
since expanding the polynomial $h_{\alpha\beta}(x)$ symbolically to
get the coefficients with respect to the monomial basis can be expensive
and memory intensive. 

Let $\tau = \max\{{\rm deg}(f), 2k, d\max_{j}\{{\rm deg}(g_j)\}\}$. 
Then for $k$ fixed and for each $d$, we get
\begin{eqnarray}
\nonumber
q_d^k &=& \sup\Big\{ t \mid f(x) - t -\sum_{(\alp,\beta)\in\N^{2m}_{d}} \lam_{\alp\beta}\,h_{\alp\beta}(x) = \inprod{Q}{v_k(x)v_k(x)^T};
\\[5pt]
&&\hspace{6cm} Q \in \cS^{s(k)}_+,\lam\geq 0\Big\}
\label{first} \\[5pt]
&=& \sup\left\{ t\; \Big| 
\begin{array}{l}
f(x^{(p)}) = t +\sum_{(\alp,\beta)\in\N^{2m}_{d}} \lam_{\alp\beta}\,h_{\alp\beta}(x^{(p)}) 
+ \inprod{Q}{v_k(x^{(p)})v_k(x^{(p)})^T}, 
\\[5pt] 
p=1,\dots,L, \; Q \in \cS^{s(k)}_+,\; \lam \geq 0,\; t\in\R
\end{array}
\right\}
\label{eq-1}
\end{eqnarray}
where $L:=|\N^{n}_{\tau}|={n+\tau\choose n}$ and
$\{ x^{(p)}\in \R^n \mid p =1,\dots, L\}$ are randomly selected points
in $[-1,1]^n$; 
$s(k) = {n+k\choose k}$, and $v_k(x)$ is a vector 
of polynomial basis for $\R[x]_k$, the space of polynomials of degree at most $k$. 
To be rigorous, if the optimal value $q_d^k$ of (\ref{first}) is finite then  the above equality holds with probability one and every optimal solution $(q^k_d,\lambda^*,Q^*)$ of (\ref{first}) is also an optimal solution of (\ref{eq-1}).

\subsection{Sufficient condition for finite convergence}
\label{subsec-sufficient-condition}

By looking at the dual of the semidefinite program (\ref{eq-1}) one obtains a sufficient condition for finite convergence.
To describe the dual of  the semidefinite program (\ref{eq-1}) we need 
to introduce some notation.

For every $p=1,\ldots,L$, denote by $\delta_{x^{(p)}}$ the Dirac measure at the point $x^{(p)}\in \R$
and let $\langle q,\delta_{x^{(p)}}\rangle =q(x^{(p)})$ for all $p=1,\ldots,L$, and all $q\in\R[x]$.

With a real sequence $\y=(y_\alpha)$, $\alpha\in\N^n_{2\ell}$, denote by 
$\M_\ell(\y)$ the moment matrix associated with $\y$. It is a real symmetric matrix with rows and columns indexed by $\N^n_\ell$, and with entries
\[\M_\ell(\y)(\alpha,\beta)\,=\,y_{\alpha+\beta},\qquad \forall\,\alpha,\beta\in\N^n_\ell.\]
Similarly, for $j=1,\ldots,m$, letting $g_j(x)=\sum_\gamma g_{j\gamma}x^\gamma$, denote by $\M_\ell(g_j\,\y)$ the localizing matrix associated 
with $\y$ and $g_i\in\R[x]$. Its rows and columns are also indexed by $\N^n_\ell$, and
with entries
\[\M_\ell(g_j\,\y)(\alpha,\beta)\,=\,\sum_{\gamma\in\N^n}g_{j\gamma}\,y_{\alpha+\beta+\gamma},\qquad \forall\,\alpha,\beta\in\N^n_\ell.\]
The dual of  the semidefinite program (\ref{eq-1}) reads:
\begin{equation}
\label{aux-dual}
\begin{array}{rl}\tilde{q}_d^k := \displaystyle\inf_{\theta\in\R^L}&\displaystyle\sum_{p=1}^L \theta_p\,\langle f,\delta_{x^{(p)}}\rangle\\
\mbox{s.t.} &  \displaystyle\sum_{p=1}^L\theta_p\,(v_k(x^{(p)})\,v_k(x^{(p)})^T)\succeq0\\
&\displaystyle\sum_{p=1}^L \theta_p\,\langle h_{\alpha\beta},\delta_{x^{(p)}}\rangle \geq 0,\quad (\alpha,\beta)\in\N^{2m}_d\\
&\displaystyle\sum_{p=1}^L \theta_p=1.
\end{array}
\end{equation}
(Notice that the weights $\theta_p$ are not required to be nonnegative.) By standard weak duality in convex optimization, and for every fixed $k\in\N$, one has
\[f^*\,\geq\,\tilde{q}^k_d\,\geq\,q^k_d,\qquad \forall d\in\N.\]
(In full rigor, for each fixed $d$ the above inequality holds with probability one.)
Next, let $s\in\N$ be the smallest integer such that $2s\geq \max[{\rm deg}(f);\,{\rm deg}(g_j)]$, and let
$r:=\max_j\lceil{\rm deg}(g_j)/2\rceil$.
We have the following {\it verification} lemma.
\begin{lemma}
\label{lemma-dual}
Let $K$ in (\ref{setk}) be compact with nonempty interior and assume that there exists $x_0\in K$ such that
$0<g_j(x_0)<1$ for all $j=1,\ldots,m$.

(a) For every $d$ sufficiently large (say $d\geq d_0$), the semidefinite program (\ref{first}) has an optimal solution.
Hence for each fixed $d\geq d_0$, with probability one the semidefinite program (\ref{eq-1}) has also an optimal solution.

(b) Let $\theta^*\in\R^L$ be an optimal solution of (\ref{aux-dual}) (whenever it exists) and let $\y^*=(y^*_\alpha)$, $\alpha\in\N^n_{2s}$, with
\begin{equation}
\label{y-star}
y^*_\alpha\,:=\,\sum_{p=1}^L \theta^*_p\,(x^{(p)})^\alpha,\qquad \alpha\in\N^n_{2s}.\end{equation}
\indent
$\bullet$ If ${\rm rank}\,\M_s(\y^*)=1$ then $\tilde{q}^k_d=f^*$ and
$x^*=(y^*_\alpha)$, $\vert\alpha\vert=1$, i.e.,
$x^*=\sum_{p=1}^L\theta^*_p\,x^{(p)}$,
is an optimal solution of problem $(P)$.

$\bullet$ If $\M_s(\y^*)\succeq0$, $\M_{s-r}(g_j\,\y^*)\succeq0$, $j=1,\ldots,m$, and 
${\rm rank}\,\M_s(\y^*)={\rm rank}\,\M_{s-r}(\y^*)$,
then $\tilde{q}^k_d=f^*$ and problem $(P)$ has ${\rm rank}\,\M_{s}(\y^*)$ global minimizers that can be extracted by a linear algebra procedure.
\end{lemma}
The proof is postponed to the Appendix.
Notice that we do not claim that problem (\ref{aux-dual}) has always an optimal solution. 
Lemma \ref{lemma-dual} is a {\it verification lemma} (or a stopping criterion) based on some sufficient rank condition on $\M_s(\y^*)$ and $\M_{s-r}(\y)$, provided that an optimal solution $\y^*$ exists. When the latter condition holds true
then $f^*=\tilde{q}^k_d$ and we can stop as at least one global optimal solution $x^*\in K$ has been identified.

\subsection{On the rank-one matrices of (\ref{aux_{1}}) and SDPT3}

Note that in the SDP \eqref{eq-1}, the constraint matrices associated with $Q$ 
are all dense rank-1 matrices of the form $A_p = v_k(x^{(p)})v_k(x^{(p)})^T$. Thus
if we let $\bv_p = v_k(x^{(p)})$, then 
the linear maps involved in the equality constraints of the SDP can be 
evaluated cheaply based on the following formulas:
$$
 \cA(X) := \Big[ \inprod{A_p}{X} \Big]_{p=1}^L = \Big[ \inprod{\bv_p}{ X  \bv_p}\Big]_{p=1}^L,\quad
 \cA^\ast y := \sum_{y=1}^{L} y_p A_p = V\,{\rm Diag}(y) V^T
$$
where $X\in \cS^{s(k)}$, $y\in \R^L$, $V=[\bv_1,\dots,\bv_L]\in \R^{s(k)\times L}$.
Moreover, one need not store the dense constraint 
matrices $\{ A_p \mid p=1,\dots,L\}$ but only the vectors
$\{ \bv_p\mid p=1,\dots,L\}.$
{\it To solve the SDP \eqref{eq-1} efficiently, we need to exploit
the rank-$1$ structure of the constraint matrices during the iterations.}
Fortunately, the SDPT3 solver \cite{TTT99,TTT03} based on interior point methods
has already been designed to exploit such a rank-1 structure
to minimize the memory needed to store the constraint matrices,
as well as to minimize the computational cost required to 
compute the Schur complement matrix arising in each interior-point
iteration. 
More precisely, in each iteration where a positive definite matrix
$W\in \cS^{s(k)}$ is given, one needs to compute the Schur 
complement matrix ${\bf S}$ whose $(p,q)$ element is given by
$$
{\bf S}_{pq} = \inprod{A_p}{WA_q W} = \inprod{\bv_p\bv_p^T}{W\bv_q\bv_q^TW}
 = \inprod{\bv_p}{W\bv_q}^2, \quad p,q = 1,\dots,L.
$$

It is the combination of these two implementation techniques
(point evaluation in the formulation and exploiting rank-one structure in the 
interior point algorithm) that makes 
our implementation of the SOS-hierarchy (\ref{aux_{1}}) efficient.

\section{Computational issues}

Given $f\in \R[x]_d$, 
in order to efficiently evaluate the vector $f(x^{(p)})$, $p=1,\dots, L$, we need a
convenient representation of the polynomial $f(x)$. In our implementation of 
BSOS, we use the following
data format to input a polynomial: 
\begin{eqnarray*}
 \bF(i,1:n+1) = [\alpha^T, f_\alpha]
\end{eqnarray*}
where $f_\alpha$ is the $i$th coefficient corresponding to the 
monomial $x^\alpha$. Note that the enumeration of the coefficients 
of $f(x)$ is not important. For a given point $z\in \R^n$ such that $z_i\not=0$ for all $i=1,\dots,n$, we evaluate $f(z)$
via the following procedure written in {\sc Matlab} syntax:
\begin{description}
\item[Step 1.] Set
$\bP = \bF(:,1:n)$, $\ff=\bF(:,n+1)$, and
  $s = (s_1,\dots, s_n)^T$, where
$s_i = 1$ if $z_i < 0$, and $s_i=0$ if $z_i \geq 0$. 
\item[Step 2.] Compute $\bar{s} = {\tt rem}(\bP s,2)$ and 
$\bz=\exp(\bP\log|z|)$.
\item[Step 3.] Compute
$f(z) = \inprod{\ff^{(a)}}{\bz^{(a)}} - \inprod{\ff^{(b)}}{ \bz^{(b)}}$,
where $\ff^{(a)} =\ff({\tt find}(\bar{s}==0))$ and 
$\ff^{(b)} =\ff({\tt find}(\bar{s}==1)).$
\end{description}
(The above procedure can be modified slightly to handle the case when $z$ has 
some zero components.)
Note that in the above procedure, 
$ \inprod{\ff^{(a)}}{\bz^{(a)}}$ and $ \inprod{\ff^{(b)}}{ \bz^{(b)}} $
correspond to the sum of positive terms and sum of negative terms in the evaluation
of $f(z)$. By separating the summation of the positive and negative 
terms in the evaluation of $f(z)$, it is hoped that cancellation errors
can be minimized. 

We should mention that some of the equality constraints in \eqref{eq-1}
may be redundant. For the sake of reducing the computational cost and
improve the numerical stability, we remove these redundant constraints 
before solving the SDP. However, as $d$ increases, the linear constraints
would become more and more nearly dependent, and typically
the SDP problem cannot be solved accurately by either SDPT3 or SEDUMI.

Another numerical issue which we should point out is that 
the constraint matrix 
$$
\left[\begin{array}{c}
\big(h_{\alpha\beta}(x^{(1)})\big)_{(\alpha,\beta)\in \N^{2m}_d} \\
\vdots \\ 
\big(h_{\alpha\beta}(x^{(L)})\big)_{(\alpha,\beta)\in \N^{2m}_d}
\end{array}\right]
$$
associated with the nonnegative vector $(\lam_{\alpha\beta})$ is typically fully dense.
Such a matrix would consume too much memory and also computational
cost when $d$ increases or when $m$ is large.


\section{Numerical experiments}

We call our approach BSOS (for hierarchy with bounded degree SOS).
As mentioned in the Introduction, we conduct experiments 
on three classes of problems which will be described
in the ensuing subsections.

\subsection{Comparison of BSOS with Gloptiploy }

We construct a set of test functions with 5 constraints. 
The test functions are mainly generated based on the following two problems:
\begin{eqnarray*}
 \begin{array}{rrrrrrrrrl}
 (P_1)         
&    &           f =&x_1^2 &- x_2^2   &              &+x_3^2   &- x_4^2  & +x_1  &-x_2\\[5pt]
\mbox{s.t.} 
&0 &\leq g_1=&2x_1^2 &+3x_2^2 &+2x_1x_2 &+2x_3^2 &+3x_4^2 &+2x_3x_4 & \leq 1\\[3pt]
&0 &\leq g_2=&3x_1^2 &+2x_2^2 &-4x_1x_2 &+3x_3^2  &+2x_4^2 &-4x_3x_4 &\leq 1\\[3pt]
&0 &\leq g_3=&x_1^2  &+6x_2^2 &-4x_1x_2 &+x_3^2   &+6x_4^2 &-4x_3x_4 &\leq 1\\[3pt]
&0 &\leq g_4=&x_1^2  &+4x_2^2 &-3x_1x_2 &+x_3^2    &+4x_4^2 &-3x_3x_4 &\leq 1\\[3pt]
&0 &\leq g_5=&2x_1^2&+5x_2^2 &+3x_1x_2 &+2x_3^2 &+5x_4^2 &+3x_3x_4 &\leq 1\\[3pt]
&0& \leq x. \quad
\end{array}
\end{eqnarray*}
The optimal value of $(P_1)$ is $f(x^*)=-0.57491$, as computed by GloptiPoly3. 
For BSOS, we get the result $q^{k=1}_{d=1} = -0.57491$, 
which is the exact result.

The second problem is : 
\begin{eqnarray*}
&&
\begin{array}{rrrrrrl}
(P_2)  
&  &   f           =&x_{1}^4x_{2}^2 &+x_{1}^2x_{2}^4 &-x_{1}^2x_{2}^2 \\[5pt]
\mbox{s.t.}
&0 &\leq g_1 =&x_{1}^2    &+x_{2}^2    &                   &\leq 1\\[3pt]
&0 &\leq g_2=&3x_{1}^2   &+2x_{2}^2  &-4x_{1}x_{2}&\leq 1\\[3pt]
&0 &\leq g_3=&x_{1}^2    &+6x_{2}^4   &-8x_{1}x_{2}+2.5&\leq 1\\[3pt]
&0 &\leq g_4=&x_{1}^4    &+3x_{2}^4  &                  & \leq 1\\[3pt]
&0 &\leq g_5=&x_{1}^2    &+x_{2}^3   &                   & \leq 1 
\end{array}
\\[3pt]
&&\qquad\quad\; 0 \leq x_1,\quad 0 \leq x_2. 
\end{eqnarray*}
The optimal value of $(P_2)$ is $f(x^*)=-0.037037$, as computed by GloptiPoly3. 
The results obtained by BSOS are
\begin{eqnarray*}
\begin{array}{lll}
q^{k=3}_{d=1} = -0.041855, &q^{k=3}_{d=2} = -0.037139, 
&q^{k=3}_{d=3} =  -0.037087 \\[5pt]
q^{k=3}_{d=4} =  -0.037073, & q^{k=3}_{d=5} = -0.037046\\[5pt]
q^{k=4}_{d=1} = -0.038596, &q^{k=4}_{d=2} = -0.037046, 
&q^{k=4}_{d=3} =  -0.037040\\[5pt]
q^{k=4}_{d=4} =-0.037038, & q^{k=4}_{d=5} = -0.037037.
\end{array}
\end{eqnarray*}
Based on the above two problems, we increase the degree of the objective function and constraint functions to generate other test instances which are 
given explicitly in the Appendix. 

Table \ref{table1} compares the results obtained by 
 BSOS and GloptiPoly3 for the tested instances. 
We observe that BSOS can give the exact result for those problems with either low degree or low dimension, while also providing a good lower bound for high degree and high dimensional problems. 
In particular on this sample of problems, $k$ is chosen so that
the size of the semidefinite constraint (which is ${n+k\choose k}$) is the same as the one needed in GloptiPoly,
 to certify global optimality. Then notice that BSOS succeeds in finding the optimal value even though the positivity certificate used in (\ref{eq-1}) is not Putinar's certificate (\ref{put-cert}) used in GloptiPoly. 
 In addition, for most of test problems, BSOS can usually get better bounds
as $d$  increases, and in most cases, the bound is good enough for small $d=2,3$. 

In Table  \ref{table1}, we also use the sufficient condition 
stated in Lemma \ref{lemma-dual} to check whether the generated
lower bound is indeed optimal. For quite a number of instances, 
 the moment matrix $\M_\ell(\y^*)$ associated with 
the optimal solution $\theta^*$ of \eqref{aux-dual} indeed has numerical rank equal to one (we declare that the matrix has numerical rank equal to one if the largest
eigenvalue is at least $10^4$ times larger than the second largest eigenvalue), 
which certifies that the lower bound is actually the optimal value.
We should note that for some of the instances, although the lower bound
is actually the optimal value (as declared by GloptiPoly), but the rank of the 
moment matrix $\M_\ell(\y^*)$ is larger than one.

\begin{table} 
\begin{center}
\begin{scriptsize}
\caption{Comparison of BSOS and GloptiPoly3. An entry marked with 
``$*$" means that the correspond SDP was not solved to high accuracy. \label{table1}}
\begin{tabular}{ | c ||  c |c| c|@{}c@{}|| c | c |@{}c@{}|}
    \hline
     Problem
&\multicolumn{4}{|c||}{BSOS} & \multicolumn{3}{|c|}{GloptiPoly3} \\  \cline{2-8} 
    & $(d,k)$   &Result & Time(s) &rank$(\M(\y^*))$
 &Result & Time(s) &Order,Optimal\\ \hline

       {\tt P4\_2}     &1,1 &  -6.7747e-001   & 0.3 &1
&  -6.7747e-001 & 0.2 &1,yes\\  
                             &2,1& -6.7747e-001   & 0.5  & 1 &&&
\\     \hline
       {\tt P4\_4}     & 1,2 &-2.9812e-001  & 0.5   &7
& -3.3539e-002 & 0.3 &2,yes\\ 
                             & 2,2 & -3.3539e-002 & 0.6   & 4 &&&
 \\  \hline

     {\tt P4\_6}    & 1,3 &-6.2500e-002   & 0.8    &31
&  -6.0693e-002 & 0.5&3,yes \\
                             & 2,3 & -6.0937e-002 & 0.9   &7 &&&\\
                             & 3,3 & -6.0693e-002 &1.8    &4 &&&
\\ \hline
    {\tt P4\_8}    & 1,4 &-9.3354e-002$*$  &  3.2   & $>10$
& -8.5813e-002  & 2.6 &4,yes\\  
                             & 2,4& -8.5813e-002   &  3.7   &9  &&&\\  
                             & 3,4 &-8.5813e-002   &  5.1   &4  &&&
\\    \hline
       {\tt P6\_2}    & 1,1 & -5.7491e-001         &  0.3  &1
 &   -5.7491e-001   & 0.2 &1,yes\\
                             & 2,1 & -5.7491e-001         &  0.8  &1 &&&
\\ \hline
    {\tt P6\_4}    & 1,2 &-5.7716e-001    &   1.1 &10
 &  -5.7696e-001  & 0.3  &2,yes\\ 
                             & 2,2& -5.7696e-001    &   1.1  &4  &&&\\ 
                             & 3,2 & -5.7696e-001    &  4.3  &1 &&&\\
 \hline
    {\tt P6\_6}     & 1,3 &-6.5972e-001  & 7.1  &  $>10$
&  -4.1288e-001  & 6.4 &3,yes\\
                             & 2,3 &-6.5972e-001  &10.2  & $>10$   &&&  \\
                             & 3,3 &-4.1288e-001  &32.0  &1 &&&  \\
\hline
     {\tt P6\_8}     & 1,4 &-6.5973e-001  &  74.2 & $>10$
&-4.0902e-001    &207.2   &  4,yes\\ 
                             & 2,4 & -6.5973e-001 &168.6  & $>10$  &&& \\
                             & 3,4 & -6.5973e-001  &264.1  & $>10$ &&& \\
                             & 4,4 & -4.0928e-001$*$  &1656.0 & 1$^*$  &&& 
\\ \hline
       8 var, deg 2    & 1,1& -5.7491e-001 &0.5  &1
& -5.7491e-001 & 0.3 &1,yes\\ 
                             & 2,1&-5.7491e-001 &0.9 & 1 &&& 
\\  \hline
       8 var, deg 4    & 1,2 &-6.5946e-001 &2.8  & $>10$
&-4.3603e-001 &  1.5  &2,yes\\ 
                             & 2,2 &-4.3603e-001 &4.8   &1 &&&
\\\hline
       8 var, deg 6    & 1,3&-6.5973e-001 &  127.1  & $>10$
&  -4.1288e-001 &  161.3&3,yes  \\ 
                             & 2,3&-6.5973e-001  &126.6  & $>10$  &&& \\                                 
                             & 3,3&-4.1322e-001*  &258.7 &1$^*$  &&&
\\ \hline
      10 var, deg 2   & 1,1&-5.7491e-001  &0.4  &1
&-5.7491e-001  & 0.2  &1,yes\\
                            & 2,1&-5.7491e-001  &1.0 &1  &&& 
\\  \hline
      10 var, deg 4   & 1,2 &-6.5951e-001  &  7.8   &1
& -4.3603e-001  &5.3   &2,yes \\  
                             & 2,2 &-4.3603e-001   &20.0   &1 & & & \\
                             & 3,2 &-4.3603e-001$*$ &66.7 &1$^*$  & & & 
\\\hline
      20 var, deg 2    & 1,1& -5.7491e-001  &  1.2  &1
& -5.7491e-001   & 0.4 &1,yes\\ 
                              &2,1& -5.7491e-001  &  3.0 & 1  &&&
\\   \hline      
      20 var, deg 4    & 1,2&infeasible & 302.1    & -
& -4.3603e-001   & 5600.8 &2,yes\\ 
                              &2,2& -4.3602e-001*  &1942.2   &1$^*$ &&&
\\   \hline      
\end{tabular}
\end{scriptsize}
\end{center} 
\end{table}

\subsection{Comparison of BSOS with the LP relaxations of Krivine-Stengle 
on convex problems}

Here we compare the performance of BSOS with the LP relaxations of Krivine-Stengle on convex problems where each test problem has 5 constraint functions in addition
to the nonnegative constraint $x\geq 0$.
Note that the LP relaxation problem has exactly the same form as in \eqref{eq-1},
except that the positive semidefinite matrix variable $Q$ is set to $0$. 
We should mention that even though the Krivine-Stengle scheme 
generates LP problems instead of SDP problems, the size of the 
corresponding LP problems also increases rapidly with $d$, like for the BSOS scheme. In particular,
in both LP- and BSOS-relaxations, the dimension of the nonnegative variable $\lambda$ is 
${2m+d \choose d}$, and the constraint
matrix is fully dense. (The BSOS-relaxations include an additional semidefinite constraint with
fixed matrix size ${n+k\choose k}$.)
The following example illustrates the performance of LP relaxation method:
\begin{eqnarray*}
\begin{array}{rlrrrl}
(C_1)     \quad   
\min 
&f               &= &x_1^4    &+x_2^4 &+2x_1^2x_2^2-x_1-x_2\\[5pt]
\mbox{s.t.} 
& 0 \leq g_1&=&-x_1^4  &-2x_2^4 &+1 \\[3pt]
&0 \leq g_2 &=&-2x_1^4 &-x_2^4&+1 \\[3pt]
&0 \leq g_3 &=&-x_1^4  &-4x_2^2&+1.25 \\[3pt]
&0 \leq g_4 &=&-4x_1^4  &-x_2^4&+1.25 \\[3pt]
&0 \leq g_5 &=&-2x_1^4 &-3x_2^2&+1.1 \\[3pt]
& 0\leq x_1\\[3pt]
& 0 \leq  x_2.
\end{array}
\end{eqnarray*}
For this problem, the functions $f$ and $-g_i$'s are all convex.
The optimal value for this problem is $f(x^*)=-0.7500$, as computed by 
GloptiPoly3. For BSOS, we get 
$q^{k=2}_{d=1}=-0.7500$, and we obtained the exact result by just choosing $d=1$.
This observation is consistent with Theorem 4.1 in \cite{lasserre-new}.
For the LP relaxation method, we get 
the following values for various choices of $d$:
\begin{eqnarray*}
q^{LP}_{d=1} = \mbox{infeasible},\;\; q^{LP}_{d=2} = -1.2200, \;\;
q^{LP}_{d=3} =-1.0944,\;\;
q^{LP}_{d=4} =-0.9696,\;\;
q^{LP}_{d=5} =\mbox{fail.}
\end{eqnarray*}
Observe that when  $d$ increases, we could get a better lower bound for the exact optimal value.
However, as $d$ increases, the  LP relaxation problem would
become increasing ill-posed and the solver has difficulty in solving 
LP problem accurately. 
In particular, for $d=5$, both the solvers SeDuMi and SDPT3 fail to 
compute an accurate enough solution for the LP to generate 
a sensible lower bound for $f(x^*)$. 

In Table \ref{table2}, we observe that 
BSOS can achieve the exact result with $d=1$ for all the test instances.
In contrast, the LP relaxation method of Krivine-Stengle does not perform very well 
even though the test instances are convex problems. In particular,
observe that for the last instance {\tt C20\_2}, the LP relaxation method cannot 
produce a good lower bound even when we choose $d=3$, and the 
time taken to solve the correspond LP is about $40$ minutes.

\begin{center}
\begin{scriptsize}
\begin{longtable}{ |@{}r@{} || r |@{}r@{}|r||r|@{}r@{}|r ||@{}r@{}|r|@{}c@{}|} 
\caption{Comparison of BSOS with LP relaxations of Krivine-Stengle on convex problems.
\label{table2}}
 \\   \hline
    & \multicolumn{3}{| c || }{LP} & \multicolumn{3}{| c || }{BSOS}& \multicolumn{3}{| c | }{GloptiPoly3} \\ \cline{2-10} 
      &$d$ &Result &Time(s)   &$d,k$ &Result  &Time(s) & Result & Time(s)& Order,Optimal
\\ \hline
\endhead

       {\tt C4\_2}  &   $1$  & infeasible     &     
&$1,1$    &  -2.5000e-001     &0.4 
&  -2.5000e-001   &0.2& 1,yes\\ 
                            &   $2$  & -9.0000e-001      &   0.1    &   &  &&&&\\
                            &   $3$  & -5.8852e-001      &   0.3    &   &  &&&&\\
                            &   $4$  & -4.2500e-001      &   5.6    &   &  &&&&\\
                            &   $5$  & -3.4975e-001      &   98.4    &   &  &&&&\\
                            &   $6$  & -3.1001e-001     & 4074.1     &   &  &&&&
\\     \hline
      {\tt C4\_4}     &   $\leq3$  &infeasible      &       &$1,2$   
& -6.9574e-001  & 0.6
& -6.9574e-001 & 0.2&2,yes \\                              
                          &   $4$  & -1.1094e+000      & 16.9     &   &  &&&&\\
                          &   $5$  & -8.8542e-001      & 788.4   &   &  &&&&
\\    \hline
       {\tt C4\_6}     &   $\leq 5$  &infeasible          &       &  $1,3$
&-1.1933e+000   & 1.5
&-1.1933e+000   & 0.5 &3,unknown\\ 
                            &   $6$  &fail          &      & &  &&&&
\\   \hline

       {\tt C6\_2}     &   $1$  & infeasible     &       &$1,1$   
&-2.5000e-001    & 0.3
&-2.5000e-001    & 0.2& 1,yes\\ 
                          &   $2$  & -9.0000e-001      &    0.1   &   &  &&&&\\
                          &   $3$   & -5.8852e-001      &    0.6   &   &  &&&&\\
                          &   $4$   & -4.2500e-001      & 66.3    &   &  &&&&\\
                          &   $5$   & -3.4975e-001       &4069.3     &   &  &&&&
\\     \hline
       {\tt C6\_4}   &   $\leq 3$  & infeasible      &       &$1,2$  
& -6.9574e-001  & 1.3 
& -6.9574e-001  & 0.4&2,yes\\ 
                            &   $4$  &-1.1094e+000       & 177.5      &    & &&&&
\\     \hline
       {\tt C6\_6}   &   $\leq 5$  & infeasible         &       &$1,3$   
& -1.1933e+000  &1.3
& -1.1933e+000 &0.4  &3,unknown \\ 
                            &   $6$  & out of memory         &     & &  &&&&
\\   \hline
%

       {\tt C8\_2}   &   $1$  & infeasible     &       &$1,1$    
& -2.5000e-001   & 0.4
& -2.5000e-001  &0.2 &1,yes \\ 
                         &   $2$  & -9.0000e-001     & 0.1      &  &  & &&&\\
                         &   $3$  & -5.8852e-001     & 3.5      &  &  & &&&\\
                        &   $4$   & -4.2500e-001     & 508.8      &  &  & &&&
\\   \hline
      {\tt C8\_4}   &   $\leq3$  & infeasible     &      &$1,2$      
&-6.9574e-001  &3.3 
&-6.9574e-001   &1.2 & 2,yes\\ 
                            &   $4$  & -1.1094e+000      & 1167.3      &    &  &&&&
\\   \hline
      {\tt C10\_2}  &   $1$  &  infeasible     &  &$1,1$  
&  -2.5000e-001   & 0.8 
&  -2.5000e-001   & 0.3&1,yes\\ 
                           &   $2$  & -9.0000e-001    & 0.1      &   &  &&&&\\
  &   $3$  &  -5.8852e-001       & 15.8      &   &  &&&&\\
  &   $4$  &  -4.2500e-001      & 9993.0      &   &  &&&&
\\   \hline
      {\tt C10\_4}  &   $\leq 3$  & infeasible     &       &$1,2$      
& -6.9574e-001  &11.4
& -6.9574e-001  &5.5& 2,yes \\ 
                            &   $4$  & -1.1094e+000 &5544.2           &    & && &&
\\   \hline
      {\tt C20\_2}  &   $1$  &infeasible      &      &$1,1$    
& -2.5000e-001 &1.7 
& -2.5000e-001 &0.4& 1,yes \\ 
                            &   $2$  &-9.0000e-001         & 0.9      &    & && && \\
                            &   $3$  &-5.8852e-001        & 2398.0      &    & && &&
\\   \hline

\end{longtable}
\end{scriptsize}
\end{center}

\subsection{Performance of BSOS on quadratic problems with polyhedral constraints}

Here consider the following problem:
\begin{eqnarray}
 \begin{array}{rl}
  \min & x^T A x \\[5pt]
  \mbox{s.t.} & e^T x \leq 1, \quad x \geq 0, \quad x\in\R^n,
 \end{array}
\label{eq-QP}
\end{eqnarray}
where $A$ is a given $n\times n$ symmetric matrix. 
In our numerical experiments, we generate 
random instances such as {\tt Qn10\_r2} for which $n=10$ and 
$A$ is randomly generated
so that it has $r=2$ negative eigenvalues and $n-r$ positive eigenvalues as
follows: 
\begin{verbatim}
    rng('default')
    A1 = randn(n); A2 = A1*A1'; perm=randperm(n);
    [V,D] = eig(A); eigval=diag(D); idx1=perm(1:r); idx2=perm(r+1:n); 
    V1=V(:,idx1); V2=V(:,idx2); d1=eigval(idx1); d2=eigval(idx2);
    A = V2*diag(d2)*V2' - V1*diag(d1)*V1';
\end{verbatim}

Table \ref{table3} compares the performance of BSOS and 
GloptiPoly3. From the numerical results, we can see that BSOS 
is far more efficient than GloptiPoly3 in solving the problems \eqref{eq-QP}.
For example, for the problem {\tt Qn20\_r2} with $n=20$, BSOS 
took only 1.9 seconds to generate the lower bound
$-2.0356e3$ for the problem, but GloptiPoly3 took more
than 1 hour to generate the same bound. 
The disparity in the efficiency between BSOS and GloptiPoly3
is expected to become even wider
for other instances with $n$  larger than $20$. 

In Table  \ref{table3}, we again use the sufficient condition 
stated in Lemma \ref{lemma-dual} to check whether the generated
lower bound is indeed optimal. For each of the first eight instances, 
 the moment matrix $\M_\ell(\y^*)$ associated with 
the optimal solution $\theta^*$ of \eqref{aux-dual} has numerical rank equal to one (we declare that the matrix has numerical rank equal to one if the largest
eigenvalue is at least $10^4$ times larger than the second largest eigenvalue), 
which certifies that the lower bound is actually the optimal value.

\begin{table}[h] 
\begin{center}
\begin{scriptsize}
\caption{Comparison of BSOS and GloptiPoly3 on quadratic problems
with polyhedral constraints.  \label{table3}}
\begin{tabular}{ |@{}c@{}|| @{}c@{}|@{}c@{}|@{}c@{}|@{}c@{}||@{}c@{}|@{}c@{}|@{}c@{} |@{}c@{}|}
    \hline
     Problem
&\multicolumn{4}{|c||}{BSOS} & \multicolumn{4}{|c|}{GloptiPoly3} \\  \cline{2-9} 
    & $(d,k)$   &Result & Time(s)  & rank($\M_\ell(\y^*)$)
& Order &Result & Time(s) & Optimal\\ \hline

       {\tt Qn10\_r2}   &1,1 & infeasible     & 0.8    &    & 1 & infeasible & 0.1 & \\  
         $n=10,r=2$     &2,1& -2.8023e+000  & 0.5 & 1 & 2  &-2.8023e+000 & 2.8 & yes
\\     \hline
       {\tt Qn10\_r5}   &1,1 & infeasible     & 0.7      &   & 1 & infeasible & 0.1 & \\  
      $n=10,r=5$        &2,1 &-1.9685e+001   & 0.4 & 1 & 2  &-1.9685e+001 & 2.3 & yes
\\     \hline
       {\tt Qn20\_r2}   &1,1 & infeasible     &1.5      &    & 1 & infeasible &0.1  & \\  
        $n=20,r=2$       &2,1 &-2.0356e-003    &1.9  & 1 & 2  & -2.0356e-003 &4057.0  & yes 
\\     \hline
       {\tt Qn20\_r5}   &1,1 & infeasible      & 1.7    &  & 1 & infeasible &0.1  & \\  
    $n=20,r=5$          &2,1 & -1.7900e+001   &1.0 &1  & 2  &-1.7900e+001  &3587.4  & yes 
\\     \hline
       {\tt Qn40\_r4}   &1,1 & infeasible      &10.9     &  & 1 & infeasible &1.2  & \\  
    $n=40,r=4$           &2,1 & -7.0062e+000   &10.9 &1  &   &&  & 
\\     \hline
       {\tt Qn50\_r5}   &1,1 & infeasible      & 24.9   &    & 1 & infeasible &1.5  & \\  
   $n=50,r=5$            &2,1 &-5.9870e+000   &34.8 &1    &   &&  & 
\\     \hline
       {\tt Qn100\_r10}   &1,1 & infeasible    &385.8  &  & 1 & infeasible & 108.7  & \\  
$n=100,r=10$              &2,1 &-8.8502e+000     & 1617.4 & 1   &   &&  & 
\\     \hline
       {\tt Problem 2.9} in \cite{QPcollection}  
   &1,1 & infeasible    &0.9                               &           & 1 & infeasible & 0.1  & \\  
    $n=10,r=6$         &2,1 & 3.7500e-001     &0.5  & 1   & 2 &3.7500e-001  &2.7  &yes 
\\     \hline
    {\tt A=-toeplitz([0,1,1,1,1,zeros(1,10)])}    
                               &1,1 & infeasible    &1.2        &    & 1 & infeasible & 0.1  & \\  
    $n=15,r=3$         &2,1 & -8.0000e-001     &0.6 &11  & 2 & -8.0087e-001   &233.7  &unknown
\\     \hline
    {\tt A=-toeplitz([0,1,1,zeros(1,17)])}    
   &1,1 & infeasible    &1.9                                      & & 1 & infeasible & 0.2  & \\  
    $n=20,r=7$         &2,1 & -6.6667e-001     &1.3   & 18  & 2 &  -6.6905e-001 &4753.3  &unknown
\\     \hline

\end{tabular}
\end{scriptsize}
\end{center}
\end{table}
\subsection{Performance of BSOS on higher order problems with more variables}
Here we consider the following problem:
\begin{eqnarray}
 \begin{array}{rl}
\min &\sum_{|\alpha| \leq \ell}\:c_{\alpha }\,x^{\alpha}\\[5pt]
\mbox{s.t.}& x_i \geq 0,\quad i=1,...n\\[5pt]
&\sum_{i=1}^n x_i^2 \leq 1
 \end{array}
\label{eq-HP}
\end{eqnarray}
where $\ell=2$ or $\ell=4$ and $c_{\alpha}$ are randomly generated in $[-1,1]$. In our numerical experiments, we generated instances such as {\tt Hn20\_4} for which $n=20$ and the degree of the polynomial is $4$. The problem is find the minimal value of a polynomial  on the Euclidean unit ball intersected with the positive orthant.

\begin{table}[h] 
\begin{center}
\begin{scriptsize}
\caption{Comparison of BSOS and GloptiPoly3 on higher order problems with more variables. \label{table4}}
\begin{tabular}{ |@{}c@{}|| @{}c@{}|@{}c@{}|@{}c@{}|@{}c@{}||@{}c@{}|@{}c@{}|@{}c@{} |@{}c@{}|}
    \hline
     Problem
&\multicolumn{4}{|c||}{BSOS} & \multicolumn{4}{|c|}{GloptiPoly3} \\  \cline{2-9} 
    & $(d,k)$   &Result & Time(s)  & rank($\M_\ell(\y^*)$)
& Order &Result & Time(s) & Optimal\\ \hline

       {\tt Hn20\_2}   &1,1 & -2.7002e+000     & 1.1    &  2  & 1 &  -2.7002e+000 & 2.3 & unknown \\  
         $n=20,\ell=2$     &2,1& -2.3638e+000  & 1.5 & 1 & 2  & out of memory &  & 
\\     \hline
       {\tt Hn20\_4}   &1,2 & -2.1860e+000    & 254     & $>10 $  & 1 & out of memory &  & \\  
       $n=20,\ell=4$     &2,2 &-1.5943e+000   & 294 & 1 &  &  &  & 
\\     \hline
       {\tt Hn30\_2}   &1,1 & -2.1320e+000     &2.0      & 2   & 1 & -2.1320e+000 &2.5  &  unknown \\  
     $n=30,\ell=2$         &2,1 &  -1.8917e+000   &12  & 1 & 2  & out of memory  &  &  
\\     \hline
       {\tt Hn30\_4}   &1,2 & -3.1126e+000      & 418   &$ >10$ & 1 & out of memory &   & \\  
    $n=30,\ell=4$          &2,2 & -1.0781e+000   &512 &1  &   &   & &  
\\     \hline
       {\tt Hn40\_2}   &1,1 & -2.3789e+000      &4.6    & 2 & 1 & -2.3789e+000 &4.0  &  unknown \\  
    $n=40,\ell=2$           &2,1 & -2.1138e+000   &29 &1  &2   & out of memory &  & 
\\     \hline
       {\tt Hn40\_4}   &1,2 & -3.2917e+000     & 812   & $>10$   & 1 & out of memory &  & \\  
   $n=40,\ell=4$            &2,2 & -1.6531e+000 & 1102 &1    &  &  &  & 
\\     \hline

\end{tabular}
\end{scriptsize}
\end{center}
\end{table}

Table \ref{table4} displays  the respective performance of BSOS and GloptiPoly3. From the numerical results, we can see that BSOS is far more efficient than GloptiPoly3 in solving  problems \eqref{eq-HP}. For example, for  problem {\tt Hn20\_4} (with $n=20$ variables and degree $\ell=4$), BSOS took about 250s to generate the first lower bound $-2.1860$, and took nearly 300s to generate a better lower bound of $-1.5943$, which is also the exact optimal value for the problem. But GloptiPoly3 got out of memory when solving the same problem. Similarly for problem {\tt Hn40\_2}, it took BSOS and GloptiPoly3 very little time to generate the first lower bound of $-2.3789$. To get a better lower bound, it took BSOS 29s to generate the optimal value of the problem. In contrast GloptiPoly3 got out of memory for improving the bound. From our observations, the disparity in efficiency between BSOS and GloptiPoly will become wider for instances with larger $n$ and/or of higher degree.

\section{Conclusion}

We have described and tested a new hierarchy of semideifinite relaxations for global polynomial optimization.
It tries to combine some advantages of previously defined LP- and SOS-hierarchies. Essentially,
it uses a positivity certificate already used in the LP-hierarchy but with an additional semidefinite constraint 
which thus makes it an SOS-hierarchy. However the main and crucial point is that
the size of this additional semidefinite constraint is fixed in advance and decided by the user (in contrast to the 
standard SOS-hierarchy in which the size of the semidefinite constraint increases in the hierarchy). Preliminary results 
are encouraging especially for non convex problems on convex polytopes where problems with up to 
$100$ variables have been solved in a reasonable amount of time (whereas the standard SOS-hierarchy of GloptiPoly cannot be implemented). 

For problems of larger size one needs to consider some serious numerical issues 
due to the presence of some fully dense submatrix and some nearly dependent linear constraints. In addition,
to be able to handle  large-scale problems one also needs to provide a ``sparse version" of this hierarchy,
an analogue of the sparse version of the SOS-hierarchy defined in \cite{waki}. 
Both issues (a topic of further investigation) are certainly non trivial, in particular the latter issue
because the positivity certificate used in this new hierarchy involves products of initial 
polynomial constraints, which destroys the sparsity pattern considered in \cite{waki}.

\section*{Appendix}
\label{appendix}
Before proving Lemma \ref{lemma-dual} we need introduce some notation.
For $\tau\geq2$ and a sequence $\y=(y_\alpha)\in \N^n_\tau$, let $L_\y:\R[x]_\tau\to\R$ be the Riesz functional:
\[f\:\left(:=\sum_{\alpha\in\N^n_\tau}f_\alpha\,x^\alpha\,\right)\quad\mapsto \quad L_\y(f):=
\sum_{\alpha\in\N^n_\tau}f_\alpha\,y_\alpha,\qquad f\in\R[x]_\tau,\]
and let $\M_k(\y)$ be the moment matrix of order $k$, associated with $\y$. 
If $q\in\R[x]_k$ with coefficient vector $\q=(q_\alpha)$, then $\langle \q,\M_k(\y)\,\q\rangle=L_\y(q^2)$
and if $\y$ is the (truncated) moment sequence of a measure $\mu$,
\[\langle \q,\M_k(\y)\,\q\rangle\,=\,L_\y(q^2)\,=\,\int q(x)^2\,d\mu(x).\]

\subsection*{Proof of Lemma \ref{lemma-dual}}

(a) We first prove that the dual of (\ref{first}) which is the semidefinite program:
\begin{equation}
\label{two}
\rho_d^k := \displaystyle\inf_{\y\in\R^L}\,\{\,L_\y(f)\,:\:
\displaystyle\M_k(\y)\,\succeq\,0;\: L_\y(1)\,=\,1;\quad 
L_\y(h_{\alpha\beta})\, \geq 0,\quad (\alpha,\beta)\in\N^{2m}_d\,\}
\end{equation}
satisfies Slater's condition. Recall that $K$ has nonempty interior;
so let $\y$ be the sequence of
moments of the Lebesgue measure $\mu$ on $\K$, scaled to be a probability measure,
so that $L_\y(1)=1$. Necessarily
$\M_k(\y)\succ0$. Otherwise there would exists $0\neq q\in\R[x]_k$ such that 
\[\langle \q,\M_k(\y)\,\q\rangle\,=\, \int_K q(x)^2\,d\mu(x)\,=\,0.\]
But then $q$ vanishes almost everywhere on $K$, which implies $q=0$, a contradiction.

Next, observe that for each $(\alpha,\beta)\in\N^{2m}_d$, the polynomial $h_{\alpha\beta}\in\R[x]_\tau$ is nonnegative on $K$ and since there exists $x_0\in K$ such that $0<g_j(x_0)<1$ for all $j=1,\ldots,m$, there is an open set $O\subset K$ such that
$h_{\alpha\beta}(x)>0$ on $O$ for all $(\alpha,\beta)\in\N^{2m}$. Therefore
\[L_\y(h_{\alpha\beta})\,=\,\int_Kh_{\alpha\beta}\,d\mu\,\geq\,\int_Oh_{\alpha\beta}\,d\mu\,>0,\quad\forall\, (\alpha,\beta)\in\N^{2m}.\]
Therefore $\y$ is a strictly feasible solution of (\ref{two}), that is,
Slater's condition holds true for (\ref{two}). Hence $\rho^k_d=q^k_d$ for all $d$. It remains to prove that
$q^k_d>-\infty$. But this follows from Theorem \ref{th-prelim}(b) as soon as $d$ is sufficiently large, say $d\geq d_0$ for some integer $d_0$. Indeed then $-\infty<\theta_d\leq q^k_d\leq f^*$ for all $d\geq d_0$.
Finally for each fixed $d$, with probability one (\ref{first}) and (\ref{eq-1}) have same optimal value $q^k_d$, and so,
an optimal solution $(q^k_d,\lambda^*,Q^*)$ of (\ref{first}) is also an optimal solution of (\ref{eq-1}).\\

(b) Let $\theta^*$ be an optimal solution of (\ref{aux-dual}) and let $\y^*$ be as in (\ref{y-star}).\\

$\bullet$ If ${\rm rank}\,\M_s(\y^*)=1$ then $\M_s(\y^*)=v_s(x^*)\,v_s(x^*)^T$ for some $x^*\in\R^n$, and
so by definition of the moment matrix $\M_s(\y^*)$, $\y^*=(y^*_\alpha)$, $\alpha\in\N^n_{2s}$, is the vector of moments (up to order $2s$) of the Dirac measure $\delta_{x^*}$ at the point $x^*$.  That is, $y^*_\alpha=(x^*)^\alpha$ for every $\alpha\in\N^n_{2s}$.
But from (\ref{y-star}),
\[(x^*)^\alpha\,=\,y^*_\alpha\,=\,\sum_{p=1}^L\theta^*_p\,(x^{(p)})^\alpha,\quad\forall\alpha\in\N^n_{2s}.\]
In particular, for moments of order $1$ we obtain $x^*=\sum_{p=1}^L\theta^*_p\,x^{(p)}$.
In other words, up to moments of order $2s$, one cannot distinguish the Dirac measure $\delta_{x^*}$ at $x^*$ from the signed measure
$\mu=\sum_p\theta^*_p\delta_{x^{(p)}}$ (recall that the $\theta^*_p$'s are not necessarily nonnegative).
That is, $(x^*)^\alpha=\int x^\alpha d\delta_{x^*}\,=\,\int x^\alpha d\mu$ for all $\alpha\in\N^n_{2s}$. This in turn implies
that for every $q\in\R[x]_{2s}$:
\[q(x^*)\,=\,\langle q,\delta_{x^*}\rangle\,=\,\langle q,\mu\rangle\,=\,
\langle q,\sum_{p=1}^L\theta^*_p\,\delta_{x^{(p)}}\rangle\,=\,
\sum_{p=1}^L\theta^*_p\,q(x^{(p)}).\]
Next, as $\theta^*$ is feasible for (\ref{aux-dual}) and $2s\geq \max[{\rm deg}(f);\,{\rm deg}(g_j)]$,
\[0\leq\displaystyle\sum_{p=1}^L \theta^*_p\,\langle h_{\alpha\beta},\delta_{x^{(p)}}\rangle 
=\left\langle h_{\alpha\beta},\displaystyle\sum_{p=1}^L \theta^*_p\,\delta_{x^{(p)}}\right\rangle\,=\,h_{\alpha\beta}(x^*), \quad\forall (\alpha,\beta):\:{\rm deg}(h_{\alpha\beta})\,\leq 2s.\]
In particular, choosing $(\alpha,\beta)\in\N^{2m}_{2s}$ such that $h_{\alpha\beta}=g_j$ (i.e. $\beta=0$, $\alpha_i=\delta_{i=j}$),
one obtains
$g_j(x^*)\geq0$, $j=1,\ldots,m$,
 which shows that $x^*\in K$. In addition, 
\[f^*\,\geq\,\tilde{q}^k_d\,=\, \displaystyle\sum_{p=1}^L \theta^*_p\,\langle f,\delta_{x^{(p)}}\rangle\,=\,
\left\langle f,\displaystyle\sum_{p=1}^L \theta^*_p\,\delta_{x^{(p)}}\right\rangle\,=\,f(x^*),\]
  which proves that $x^*\in K$  is an optimal solution of problem $(P)$.
  
$\bullet$  If $\M_s(\y^*)\succeq0$, $\M_{s-r}(g_j\,\y^*)\succeq0$, $j=1,\ldots,m$, and
  ${\rm rank}\,\M_s(\y^*)={\rm rank}\,\M_{s-r}(\y^*)$ then by 
  Theorem \cite[Theorem 3.11, p. 66]{lasserre-imperial}, $\y^*$ is the vector of moments up to order $2s$, of some 
  atomic-probability measure $\mu$ supported on $v:={\rm rank}\,\M_s(\y^*)$ points 
  $z(i)\in K$, $i=1,\ldots,v$. That is, there exist positive weights $(w_i)\subset\R_+$ such that
  \[\mu\,=\,\sum_{i=1}^v w_i\,\delta_{z(i)};\quad \sum_{i=1}^vw_i=1;\quad w_i\,>\,0,\:i=1,\ldots,v.\]
  Therefore,
  \[f^*\,\geq\,\tilde{q}^k_d\,=\,\sum_{p=1}^L\theta^*_p\,\langle f,\delta_{x(p)}\rangle
  \,=\,\sum_{\alpha\in\N^n} f_\alpha \,y^*_\alpha\,=\,\int_K f\,d\mu\,\geq\,f^*,\]
which shows that $\tilde{q}^k_d=f^*$. In addition
 \[0\,=\,f^*- \int_K f\,d\mu\,=\,\int_K (f^*-f)\,d\mu\,=\,
 \sum_{i=1}^v \underbrace{w_i}_{>0}\,(\underbrace{f^*-f(z(i))}_{\leq 0},\]
 which implies $f(z(i))=f^*$ for every $i=1,\ldots,v$. 
 Finally, the $v$ global minimizers can be extracted from the moment matrix $\M_s(\y^*)$ by the simple linear algebra procedure described in 
 Henrion and Lasserre \cite{extraction}.   $\quad\Box$

\newpage
\section*{Appendix} 
\subsection*{Test functions for BSOS and GloptiPoly in Table \ref{table1}}
\begin{scriptsize}
Example {\tt P4\_2} (4 variables, degree 2):
\begin{eqnarray*}
\begin{array}{ll}
     f=x_{1}^2-x_{2}^2+x_{3}^2-x_{4}^2+x_1-x_2; \quad 
&g_1=2x_{1}^2+3x_{2}^2+2x_{1}x_{2}+2x_{3}^2+3x_{4}^2+2x_{3}x_{4};
\\[3pt]
g_2=3x_{1}^2+2x_{2}^2-4x_{1}x_{2}+3x_{3}^2+2x_{4}^2-4x_{3}x_{4};\quad
&g_3=x_{1}^2+6x_{2}^2-4x_{1}x_{2}+x_{3}^2+6x_{4}^2-4x_{3}x_{4};
\\[3pt]
g_4=x_{1}^2+4x_{2}^2-3x_{1}x_{2}+x_{3}^2+4x_{4}^2-3x_{3}x_{4};\quad
&g_5=2x_{1}^2+5x_{2}^2+3x_{1}x_{2}+2x_{3}^2+5x_{4}^2+3x_{3}x_{4}; \quad
x\geq 0.
\end{array}
\end{eqnarray*}
Example {\tt P4\_4} (4 variables, degree 4):
\begin{eqnarray*}
\begin{array}{ll}
f    =x_{1}^4-x_{2}^4+x_{3}^4-x_{4}^4; \quad
&g_1=2x_{1}^4+3x_{2}^2+2x_{1}x_{2}+2x_{3}^4+3x_{4}^2+2x_{3}x_{4};
\\[3pt]
g_2=3x_{1}^2+2x_{2}^2-4x_{1}x_{2}+3x_{3}^2+2x_{4}^2-4x_{3}x_{4};\quad
&g_3=x_{1}^2+6x_{2}^2-4x_{1}x_{2}+x_{3}^2+6x_{4}^2-4x_{3}x_{4};
\\[3pt]
g_4=x_{1}^2+4x_{2}^4-3x_{1}x_{2}+x_{3}^2+4x_{4}^4-3x_{3}x_{4};\quad
&g_5=2x_{1}^2+5x_{2}^2+3x_{1}x_{2}+2x_{3}^2+5x_{4}^2+3x_{3}x_{4}; \quad
x\geq 0.
\end{array}
\end{eqnarray*}
Example {\tt P4\_6} (4 variables, degree 6):
\begin{eqnarray*}
\begin{array}{ll}
    f=x_{1}^4x_{2}^2+x_{1}^2x_{2}^4-x_{1}^2x_{2}^2+x_{3}^4x_{4}^2+x_{3}^2x_{4}^4-x_{3}^2x_{4}^2; \quad
&g_1=x_{1}^2+x_{2}^2+x_{3}^2+x_{4}^2; 
\\[3pt]
g_2=3x_{1}^2+2x_{2}^2-4x_{1}x_{2}+3x_{3}^2+2x_{4}^2-4x_{3}x_{4};\quad
&g_3=x_{1}^2+6x_{2}^4-8x_{1}x_{2}+x_{3}^2+6x_{4}^4-8x_{3}x_{4}+2.5;
\\[3pt]
g_4=x_{1}^4+3x_{2}^4+x_{3}^4+3x_{4}^4;\quad
&g_5=x_{1}^2+x_{2}^3+x_{3}^2+x_{4}^3; \quad x\geq 0.
\end{array}
\end{eqnarray*}
Example {\tt P4\_8} (4 variables, degree 8):
\begin{eqnarray*}
\begin{array}{ll}
f=x_{1}^4x_{2}^2+x_{1}^2x_{2}^6-x_{1}^2x_{2}^2
+x_{3}^4x_{4}^2+x_{3}^2x_{4}^6-x_{3}^2x_{4}^2;\quad
&g_1=x_{1}^2+x_{2}^2+x_{3}^2+x_{4}^2; 
\\[3pt]
g_2=3x_{1}^2+2x_{2}^2-4x_{1}x_{2}+3x_{3}^2+2x_{4}^2-4x_{3}x_{4};\quad
&g_3=x_{1}^2+6x_{2}^4-8x_{1}x_{2}+x_{3}^2+6x_{4}^4-8x_{3}x_{4}+2.5;
\\[3pt]
g_4=x_{1}^4+3x_{2}^4+x_{3}^4+3x_{4}^4;\quad
&g_5=x_{1}^2+x_{2}^3+x_{3}^2+x_{4}^3; \quad x\geq 0.
\end{array}
\end{eqnarray*}
Example {\tt P6\_2} (6 variables, degree 2):
\begin{eqnarray*}
\begin{array}{ll}
f=x_{1}^2-x_{2}^2+x_{3}^2-x_{4}^2+x_{5}^2-x_{6}^2 + x_1-x_2;
\\[3pt]
g_1=2x_{1}^2+3x_{2}^2+2x_{1}x_{2}+2x_{3}^2+3x_{4}^2+2x_{3}x_{4}
+2x_{5}^2+3x_{6}^2+2x_{5}x_{6};
\\[3pt]
g_2=3x_{1}^2+2x_{2}^2-4x_{1}x_{2}+3x_{3}^2+2x_{4}^2-4x_{3}x_{4}
+3x_{5}^2+2x_{6}^2-4x_{5}x_{6};
\\[3pt]
g_3=x_{1}^2+6x_{2}^2-4x_{1}x_{2}+x_{3}^2+6x_{4}^2-4x_{3}x_{4}
+x_{5}^2+6x_{6}^2-4x_{5}x_{6};
\\[3pt]
g_4=x_{1}^2+4x_{2}^2-3x_{1}x_{2}+x_{3}^2+4x_{4}^2-3x_{3}x_{4}+x_{5}^2
+4x_{6}^2-3x_{5}x_{6};
\\[3pt]
g_5=2x_{1}^2+5x_{2}^2+3x_{1}x_{2}+2x_{3}^2+5x_{4}^2+3x_{3}x_{4}
+2x_{5}^2+5x_{6}^2+3x_{5}x_{6}; \quad x\geq 0.
\end{array}
\end{eqnarray*}
Example {\tt P6\_4} (6 variables, degree 4):
\begin{eqnarray*}
\begin{array}{ll}
f    =x_{1}^4-x_{2}^2+x_{3}^4-x_{4}^2+x_{5}^4-x_{6}^2 + x_1 - x_2;
\\[3pt]
g_1=2x_{1}^4+x_{2}^2+2x_{1}x_{2}+2x_{3}^4+x_{4}^2+2x_{3}x_{4}+2x_{5}^4+x_{6}^2+2x_{5}x_{6};\\[3pt]
g_2=3x_{1}^2+x_{2}^2-4x_{1}x_{2}+3x_{3}^2+x_{4}^2-4x_{3}x_{4}+3x_{5}^2+x_{6}^2-4x_{5}x_{6};\\[3pt]
g_3=x_{1}^2+6x_{2}^2-4x_{1}x_{2}+x_{3}^2+6x_{4}^2-4x_{3}x_{4}+x_{5}^2+6x_{6}^2-4x_{5}x_{6};\\[3pt]
g_4=x_{1}^2+3x_{2}^4-3x_{1}x_{2}+x_{3}^2+3x_{4}^4-3x_{3}x_{4}+x_{5}^2+3x_{6}^4-3x_{5}x_{6};\\[3pt]
g_5=2x_{1}^2+5x_{2}^2+3x_{1}x_{2}+2x_{3}^2+5x_{4}^2+3x_{3}x_{4}
+2x_{5}^2+5x_{6}^2+3x_{5}x_{6}, \quad x\geq 0.
\end{array}
\end{eqnarray*}
Example {\tt P6\_6} (6 variables, degree 6):
\begin{eqnarray*}
\begin{array}{ll}
f=x_{1}^6-x_{2}^6+x_{3}^6-x_{4}^6+x_{5}^6-x_{6}^6 +x_1-x_2;
\\[3pt]
g_1=2x_{1}^6+3x_{2}^2+2x_{1}x_{2}+2x_{3}^6+3x_{4}^2+2x_{3}x_{4}+2x_{5}^6+3x_{6}^2+2x_{5}x_{6};\\[3pt]
g_2=3x_{1}^2+2x_{2}^2-4x_{1}x_{2}+3x_{3}^2+2x_{4}^2-4x_{3}x_{4}+3x_{5}^2+2x_{6}^2-4x_{5}x_{6};\\[3pt]
g_3=x_{1}^2+6x_{2}^2-4x_{1}x_{2}+x_{3}^2+6x_{4}^2-4x_{3}x_{4}+x_{5}^2+6x_{6}^2-4x_{5}x_{6};\\[3pt]
g_4=x_{1}^2+4x_{2}^6-3x_{1}x_{2}+x_{3}^2+4x_{4}^6-3x_{3}x_{4}+x_{5}^2+4x_{6}^6-3x_{5}x_{6};\\[3pt]
g_5=2x_{1}^2+5x_{2}^2+3x_{1}x_{2}+2x_{3}^2+5x_{4}^2+3x_{3}x_{4}
+2x_{5}^2+5x_{6}^2+3x_{5}x_{6}, \quad x\geq 0.
\end{array}
\end{eqnarray*}
Example {\tt P6\_8} (6 variables, degree 8):
\begin{eqnarray*}
\begin{array}{ll}
f=x_{1}^8-x_{2}^8+x_{3}^8-x_{4}^8+x_{5}^8-x_{6}^8 + x_1-x_2;
\\[3pt]
g_1=2x_{1}^8+3x_{2}^2+2x_{1}x_{2}+2x_{3}^8+3x_{4}^2+2x_{3}x_{4}+2x_{5}^8+3x_{6}^2+2x_{5}x_{6};\\[3pt]
g_2=3x_{1}^2+2x_{2}^2-4x_{1}x_{2}+3x_{3}^2+2x_{4}^2-4x_{3}x_{4}+3x_{5}^2+2x_{6}^2-4x_{5}x_{6};\\[3pt]
g_3=x_{1}^2+6x_{2}^2-4x_{1}x_{2}+x_{3}^2+6x_{4}^2-4x_{3}x_{4}+x_{5}^2+6x_{6}^2-4x_{5}x_{6};\\[3pt]
g_4=x_{1}^2+4x_{2}^8-3x_{1}x_{2}+x_{3}^2+4x_{4}^8-3x_{3}x_{4}+x_{5}^2+4x_{6}^8-3x_{5}x_{6};\\[3pt]
g_5=2x_{1}^2+5x_{2}^2+3x_{1}x_{2}+2x_{3}^2+5x_{4}^2+3x_{3}x_{4}
+2x_{5}^2+5x_{6}^2+3x_{5}x_{6}, \quad x\geq 0.
\end{array}
\end{eqnarray*}
Example {\tt P8\_2} (8 variables, degree 2):
\begin{eqnarray*}
\begin{array}{ll}
f=x_{1}^2-x_{2}^2+x_{3}^2-x_{4}^2+x_{5}^2-x_{6}^2+x_{7}^2-x_{8}^2 + x_1-x_2;
\\[3pt]
g_1=2x_{1}^2+3x_{2}^2+2x_{1}x_{2}+2x_{3}^2+3x_{4}^2+2x_{3}x_{4}+2x_{5}^2+3x_{6}^2+2x_{5}x_{6}+2x_{7}^2+3x_{8}^2+2x_{7}x_{8};\\[3pt]
g_2=3x_{1}^2+2x_{2}^2-4x_{1}x_{2}+3x_{3}^2+2x_{4}^2-4x_{3}x_{4}+3x_{5}^2+2x_{6}^2-4x_{5}x_{6}+3x_{7}^2+2x_{8}^2-4x_{7}x_{8};\\[3pt]
g_3=x_{1}^2+6x_{2}^2-4x_{1}x_{2}+x_{3}^2+6x_{4}^2-4x_{3}x_{4}+x_{5}^2+6x_{6}^2-4x_{5}x_{6}+x_{7}^2+6x_{8}^2-4x_{7}x_{8};\\[3pt]
g_4=x_{1}^2+4x_{2}^2-3x_{1}x_{2}+x_{3}^2+4x_{4}^2-3x_{3}x_{4}+x_{5}^2+4x_{6}^2-3x_{5}x_{6}+x_{7}^2+4x_{8}^2-3x_{7}x_{8};\\[3pt]
g_5=2x_{1}^2+5x_{2}^2+3x_{1}x_{2}+2x_{3}^2+5x_{4}^2+3x_{3}x_{4}
+2x_{5}^2+5x_{6}^2+3x_{5}x_{6}+2x_{7}^2+5x_{8}^2+3x_{7}x_{8}; 
\quad x\geq 0.
\end{array}
\end{eqnarray*}
Example {\tt P8\_4} (8 variables, degree 4):
\begin{eqnarray*}
\begin{array}{ll}
f=x_{1}^4-x_{2}^4+x_{3}^4-x_{4}^4+x_{5}^4-x_{6}^4+x_{7}^4-x_{8}^4 + x_1-x_2;
\\[3pt]
g_1=2x_{1}^4+3x_{2}^2+2x_{1}x_{2}+2x_{3}^4+3x_{4}^2+2x_{3}x_{4}+2x_{5}^4+3x_{6}^2+2x_{5}x_{6}+2x_{7}^4+3x_{8}^2+2x_{7}x_{8};\\[3pt]
g_2=3x_{1}^2+2x_{2}^2-4x_{1}x_{2}+3x_{3}^2+2x_{4}^2-4x_{3}x_{4}+3x_{5}^2+2x_{6}^2-4x_{5}x_{6}+3x_{7}^2+2x_{8}^2-4x_{7}x_{8};\\[3pt]
g_3=x_{1}^2+6x_{2}^2-4x_{1}x_{2}+x_{3}^2+6x_{4}^2-4x_{3}x_{4}+x_{5}^2+6x_{6}^2-4x_{5}x_{6}+x_{7}^2+6x_{8}^2-4x_{7}x_{8};\\[3pt]
g_4=x_{1}^2+4x_{2}^4-3x_{1}x_{2}+x_{3}^2+4x_{4}^4-3x_{3}x_{4}+x_{5}^2+4x_{6}^4-3x_{5}x_{6}+x_{7}^2+4x_{8}^4-3x_{7}x_{8};\\[3pt]
g_5=2x_{1}^2+5x_{2}^2+3x_{1}x_{2}+2x_{3}^2+5x_{4}^2+3x_{3}x_{4}+2x_{5}^2
+5x_{6}^2+3x_{5}x_{6}+2x_{7}^2+5x_{8}^2+3x_{7}x_{8}, \quad x\geq 0.
\end{array}
\end{eqnarray*}
Example {\tt P8\_6} (8 variables, degree 6):
\begin{eqnarray*}
\begin{array}{ll}
f=x_{1}^6-x_{2}^6+x_{3}^6-x_{4}^6+x_{5}^6-x_{6}^6+x_{7}^6-x_{8}^6 + x_1-x_2;
\\[3pt]
g_1=2x_{1}^6+3x_{2}^2+2x_{1}x_{2}+2x_{3}^6+3x_{4}^2+2x_{3}x_{4}+2x_{5}^6+3x_{6}^2+2x_{5}x_{6}+2x_{7}^6+3x_{8}^2+2x_{7}x_{8};\\[3pt]
g_2=3x_{1}^2+2x_{2}^2-4x_{1}x_{2}+3x_{3}^2+2x_{4}^2-4x_{3}x_{4}+3x_{5}^2+2x_{6}^2-4x_{5}x_{6}+3x_{7}^2+2x_{8}^2-4x_{7}x_{8};\\[3pt]
g_3=x_{1}^2+6x_{2}^2-4x_{1}x_{2}+x_{3}^2+6x_{4}^2-4x_{3}x_{4}+x_{5}^2+6x_{6}^2-4x_{5}x_{6}+x_{7}^2+6x_{8}^2-4x_{7}x_{8};\\[3pt]
g_4=x_{1}^2+4x_{2}^6-3x_{1}x_{2}+x_{3}^2+4x_{4}^6-3x_{3}x_{4}+x_{5}^2+4x_{6}^6-3x_{5}x_{6}+x_{7}^2+4x_{8}^6-3x_{7}x_{8};\\[3pt]
g_5=2x_{1}^2+5x_{2}^2+3x_{1}x_{2}+2x_{3}^2+5x_{4}^2+3x_{3}x_{4}
+2x_{5}^2+5x_{6}^2+3x_{5}x_{6}+2x_{7}^2+5x_{8}^2+3x_{7}x_{8}, \quad x\geq 0.
\end{array}
\end{eqnarray*}
Example {\tt P10\_2} (10 variables, degree 2):
\begin{eqnarray*}
\begin{array}{ll}
f=x_{1}^2-x_{2}^2+x_{3}^2-x_{4}^2+x_{5}^2-x_{6}^2+x_{7}^2-x_{8}^2+x_{9}^2-x_{10}^2+x_1-x_2;
\\[3pt]
g_1=2x_{1}^2+3x_{2}^2+2x_{1}x_{2}+2x_{3}^2+3x_{4}^2+2x_{3}x_{4}+2x_{5}^2+3x_{6}^2+2x_{5}x_{6}
+2x_{7}^2+3x_{8}^2+2x_{7}x_{8}+2x_{9}^2+3x_{10}^2+2x_{9}x_{10};
\\[3pt]
g_2=3x_{1}^2+2x_{2}^2-4x_{1}x_{2}+3x_{3}^2+2x_{4}^2-4x_{3}x_{4}+3x_{5}^2+2x_{6}^2-4x_{5}x_{6}
+3x_{7}^2+2x_{8}^2-4x_{7}x_{8}+3x_{9}^2+2x_{10}^2-4x_{9}x_{10};
\\[3pt]
g_3=x_{1}^2+6x_{2}^2-4x_{1}x_{2}+x_{3}^2+6x_{4}^2-4x_{3}x_{4}+x_{5}^2+6x_{6}^2-4x_{5}x_{6}
+x_{7}^2+6x_{8}^2-4x_{7}x_{8}+x_{9}^2+6x_{10}^2-4x_{9}x_{10};
\\[3pt]
g_4=x_{1}^2+4x_{2}^2-3x_{1}x_{2}+x_{3}^2+4x_{4}^2-3x_{3}x_{4}+x_{5}^2+4x_{6}^2-3x_{5}x_{6}
+x_{7}^2+4x_{8}^2-3x_{7}x_{8}+x_{9}^2+4x_{10}^2-3x_{9}x_{10};
\\[3pt]
g_5=2x_{1}^2+5x_{2}^2+3x_{1}x_{2}+2x_{3}^2+5x_{4}^2+3x_{3}x_{4}+2x_{5}^2+5x_{6}^2+3x_{5}x_{6}
+2x_{7}^2+5x_{8}^2+3x_{7}x_{8}+2x_{9}^2+5x_{10}^2+3x_{9}x_{10};
\\[3pt]
x\geq 0.
\end{array}
\end{eqnarray*}
Example {\tt P10\_4} (10 variables, degree 4): 
\begin{eqnarray*}
\begin{array}{ll}
f=x_{1}^4-x_{2}^4+x_{3}^4-x_{4}^4+x_{5}^4-x_{6}^4+x_{7}^4-x_{8}^4+x_{9}^4-x_{10}^4 + x_1-x_2;
\\[3pt]
g_1=2x_{1}^4+3x_{2}^2+2x_{1}x_{2}+2x_{3}^4+3x_{4}^2+2x_{3}x_{4}+2x_{5}^4+3x_{6}^2+2x_{5}x_{6}
+2x_{7}^4+3x_{8}^2+2x_{7}x_{8}+2x_{9}^4+3x_{11}^2+2x_{9}x_{10};
\\[3pt]
g_2=3x_{1}^2+2x_{2}^2-4x_{1}x_{2}+3x_{3}^2+2x_{4}^2-4x_{3}x_{4}+3x_{5}^2+2x_{6}^2-4x_{5}x_{6}
+3x_{7}^2+2x_{8}^2-4x_{7}x_{8}+3x_{9}^2+2x_{10}^2-4x_{9}x_{10};
\\[3pt]
g_3=x_{1}^2+6x_{2}^2-4x_{1}x_{2}+x_{3}^2+6x_{4}^2-4x_{3}x_{4}+x_{5}^2+6x_{6}^2-4x_{5}x_{6}
+x_{7}^2+6x_{8}^2-4x_{7}x_{8}+x_{9}^2+6x_{10}^2-4x_{9}x_{10};
\\[3pt]
g_4=x_{1}^2+4x_{2}^4-3x_{1}x_{2}+x_{3}^2+4x_{4}^4-3x_{3}x_{4}+x_{5}^2+4x_{6}^4-3x_{5}x_{6}
+x_{7}^2+4x_{8}^4-3x_{7}x_{8}+x_{9}^2+4x_{10}^4-3x_{9}x_{10};
\\[3pt]
g_5=2x_{1}^2+5x_{2}^2+3x_{1}x_{2}+2x_{3}^2+5x_{4}^2+3x_{3}x_{4}+2x_{5}^2+5x_{6}^2+3x_{5}x_{6}
+2x_{7}^2+5x_{8}^2+3x_{7}x_{8}+2x_{9}^2+5x_{10}^2+3x_{9}x_{10};
\\[3pt]
x\geq 0.
\end{array}
\end{eqnarray*}
Example {\tt P20\_2} (20 variables, degree 2):
\begin{eqnarray*}
\begin{array}{lll}
f&=&x_{1}^2-x_{2}^2+x_{3}^2-x_{4}^2+x_{5}^2-x_{6}^2+x_{7}^2-x_{8}^2+x_{9}^2-x_{10}^2+x_{11}^2-x_{12}^2 + x_1-x_2
\\
&&+x_{13}^2-x_{14}^2+x_{15}^2-x_{16}^2+x_{17}^2-x_{18}^2+x_{19}^2-x_{20}^2;
\\[3pt]
g_1&=&2x_{1}^2+3x_{2}^2+2x_{1}x_{2}+2x_{3}^2+3x_{4}^2+2x_{3}x_{4}+2x_{5}^2+3x_{6}^2+2x_{5}x_{6}+2x_{7}^2+3x_{8}^2+2x_{7}x_{8}
\\
&&+2x_{9}^2+3x_{10}^2+2x_{9}x_{10}+2x_{11}^2+3x_{12}^2+2x_{11}x_{12}+2x_{13}^2+3x_{14}^2+2x_{13}x_{14}+2x_{15}^2+3x_{16}^2
\\
&&+2x_{15}x_{16}+2x_{17}^2+3x_{18}^2+2x_{17}x_{18}+2x_{19}^2+3x_{10}^2+2x_{20}x_{20};
\\[3pt]
g_2&=&3x_{1}^2+2x_{2}^2-4x_{1}x_{2}+3x_{3}^2+2x_{4}^2-4x_{3}x_{4}+3x_{5}^2+2x_{6}^2-4x_{5}x_{6}+3x_{7}^2+2x_{8}^2-4x_{7}x_{8}
\\
&&+3x_{9}^2+2x_{10}^2-4x_{9}x_{10}+3x_{11}^2+2x_{12}^2-4x_{11}x_{12}+3x_{13}^2+2x_{14}^2-4x_{13}x_{14}
\\
&&+3x_{15}^2+2x_{16}^2-4x_{15}x_{16}+3x_{17}^2+2x_{19}^2-4x_{18}x_{18}+3x_{19}^2+2x_{20}^2-4x_{19}x_{20};
\\[3pt]
g_3&=&x_{1}^2+6x_{2}^2-4x_{1}x_{2}+x_{3}^2+6x_{4}^2-4x_{3}x_{4}+x_{5}^2+6x_{6}^2-4x_{5}x_{6}+x_{7}^2+6x_{8}^2-4x_{7}x_{8}
\\
&&+x_{9}^2+6x_{10}^2-4x_{9}x_{10}+x_{11}^2+6x_{12}^2-4x_{11}x_{12}+x_{13}^2+6x_{14}^2-4x_{13}x_{14}
\\
&&+x_{15}^2+6x_{17}^2-4x_{16}x_{16}+x_{17}^2+6x_{18}^2-4x_{17}x_{18}+x_{19}^2+6x_{20}^2-4x_{19}x_{20};
\\[3pt]
g_4&=&x_{1}^2+4x_{2}^2-3x_{1}x_{2}+x_{3}^2+4x_{4}^2-3x_{3}x_{4}+x_{5}^2+4x_{6}^2-3x_{5}x_{6}+x_{7}^2+4x_{8}^2-3x_{7}x_{8}
\\
&&+x_{9}^2+4x_{10}^2-3x_{9}x_{10}+x_{1}^2+4x_{12}^2-3x_{11}x_{12}+x_{13}^2+4x_{14}^2-3x_{15}x_{14}
\\
&&+x_{15}^2+4x_{16}^2-3x_{15}x_{16}+x_{17}^2+4x_{18}^2-3x_{17}x_{18}+x_{19}^2+4x_{20}^2-3x_{19}x_{20};
\\[3pt]
g_5&=&2x_{1}^2+5x_{2}^2+3x_{1}x_{2}+2x_{3}^2+5x_{4}^2+3x_{3}x_{4}+2x_{5}^2+5x_{6}^2+3x_{5}x_{6}+2x_{7}^2+5x_{8}^2+3x_{7}x_{8}
\\
&&+2x_{9}^2+5x_{10}^2+3x_{9}x_{10}+2x_{11}^2+5x_{13}^2+3x_{12}x_{12}+2x_{13}^2+5x_{14}^2+3x_{13}x_{14}
\\
&&+2x_{15}^2+5x_{16}^2+3x_{15}x_{16}+2x_{17}^2+5x_{18}^2+3x_{17}x_{18}+2x_{19}^2+5x_{20}^2+3x_{19}x_{20};
\\
x\geq 0.
\end{array}
\end{eqnarray*}
Example {\tt P20\_4} (20 variables, degree 4): Same as {\tt P20\_2} except that 
$f$ is replaced by
\begin{eqnarray*}
\begin{array}{lll}
f&=&x_{1}^4-x_{2}^4+x_{3}^2-x_{4}^2+x_{5}^2-x_{6}^2+x_{7}^2-x_{8}^2+x_{9}^2-x_{10}^2+x_{11}^2-x_{12}^2 + x_1-x_2
\\
&&+x_{13}^2-x_{14}^2+x_{15}^2-x_{16}^2+x_{17}^2-x_{18}^2+x_{19}^2-x_{20}^2;
\end{array}
\end{eqnarray*}
\end{scriptsize}

\subsection{Test functions for BSOS versus LP relaxations of Krivine-Stengle 
on convex problems in Table \ref{table2}}
\begin{scriptsize}
Example {\tt C4\_2} (4 variables, degree 2): 
\begin{eqnarray*}
\begin{array}{lll}
f= x_1^2+x_2^2+x_3^2+x_4^2 + 2x_1x_2-x_1-x_2;\quad
&g_1=-x_1^2-2x_2^2-x_3^2-2x_4^2+1;\\[3pt]
g_2=-2x_1^2-x_2^2-2x_3^2-x_4^2+1;\quad
&g_3=-x_1^2-4x_2^2-x_3^2-4x_4^2+1.25;\\[3pt]
g_4=-4x_1^2-x_2^2-4x_3^2-x_4^2+1.25;\quad
&g_5=-2x_1^2-3x_2^2-2x_3^2-3x_4^2+1.1; 
\quad x\geq 0.
\end{array}
\end{eqnarray*}
Example {\tt C4\_4} (4 variables, degree 4): 
\begin{eqnarray*}
\begin{array}{lll}
f= x_1^4+x_2^4+x_3^4+x_4^4 + 3x_1^2x_2^2-x_1-x_2;\quad
&g_1=-x_1^4-2x_2^4-x_3^4-2x_4^4+1;\\[3pt]
g_2=-2x_1^4-x_2^4-2x_3^4-x_4^4+1;\quad
&g_3=-x_1^4-4x_2^4-x_3^4-4x_4^4+1.25;\\[3pt]
g_4=-4x_1^4-x_2^4-4x_3^4-x_4^4+1.25;\quad
&g_5=-2x_1^4-3x_2^2-2x_3^4-3x_4^2+1.1;
\quad x\geq 0.
\end{array}
\end{eqnarray*} 
Example {\tt C4\_6} (4 variables, degree 6):
\begin{eqnarray*}
\begin{array}{lll}
f= x_1^6+x_2^6+x_3^6+x_4^6 + \frac{10}{3} x_1^3x_2^3-x_1-x_2;\quad
&g_1=-x_1^6-2x_2^6-x_3^6-2x_4^6+1;\\[3pt]
g_2=-2x_1^6-x_2^6-2x_3^6-x_4^6+1;\quad
&g_3=-x_1^6-4x_2^2-x_3^6-4x_4^2+1.25;\\[3pt]
g_4=-4x_1^6-x_2^2-4x_3^6-x_4^2+1.25;\quad
&g_5=-2x_1^2-3x_2^6-2x_3^2-3x_4^6+1.1;
\quad x\geq 0.
\end{array}
\end{eqnarray*} 
Example {\tt C6\_2} (6 variables, degree 2): 
\begin{eqnarray*}
\begin{array}{lll}
f= x_1^2+x_2^2+x_3^2+x_4^2+x_5^2+x_6^2 + 2x_1x_2-x_1-x_2;\quad
&g_1=-x_1^2-2x_2^2-x_3^2-2x_4^2-x_5^2-2x_6^2+1;\\[3pt]
g_2=-2x_1^2-x_2^2-2x_3^2-x_4^2-2x_5^2-x_6^2+1;\quad
&g_3=-x_1^2-4x_2^2-x_3^2-4x_4^2-x_5^2-4x_6^2+1.25;\\[3pt]
g_4=-4x_1^2-x_2^2-4x_3^2-x_4^2-4x_5^2-x_6^2+1.25;\quad
&g_5=-2x_1^2-3x_2^2-2x_3^2-3x_4^2-2x_5^2-3x_6^2+1.1;
\quad x\geq 0.
\end{array}
\end{eqnarray*}
Example {\tt C6\_4} (6 variables, degree 4): 
\begin{eqnarray*}
\begin{array}{lll}
f= x_1^4+x_2^4+x_3^4+x_4^4+x_5^4+x_6^4+3x_1^2x_2^2-x_1-x_2;\quad
&g_1=-x_1^4-2x_2^4-x_3^4-2x_4^4-x_5^4-2x_6^4+1;\\[3pt]
g_2=-2x_1^4-x_2^4-2x_3^4-x_4^4-2x_5^4-x_6^4+1;\quad
&g_3=-x_1^4-4x_2^4-x_3^4-4x_4^4-x_5^4-4x_6^4+1.25;\\[3pt]
g_4=-4x_1^4-x_2^4-4x_3^4-x_4^4-4x_5^4-x_6^4+1.25;\quad
&g_5=-2x_1^4-3x_2^2-2x_3^4-3x_4^2-2x_5^4-3x_6^2+1.1;
\quad x\geq 0.
\end{array}
\end{eqnarray*}
Example {\tt C6\_6} (6 variables, degree 6): 
\begin{eqnarray*}
\begin{array}{lll}
f= x_1^6+x_2^6+x_3^6+x_4^6+x_5^6+x_6^6+\frac{10}{3}x_1^2x_2^3-x_1-x_2;\quad
&g_1=-x_1^6-2x_2^6-x_3^6-2x_4^6-x_5^6-2x_6^6+1;\\[3pt]
g_2=-2x_1^6-x_2^6-2x_3^6-x_4^6-2x_5^6-x_6^6+1;\quad
&g_3=-x_1^6-4x_2^2-x_3^6-4x_4^2-x_5^6-4x_6^2+1.25;\\[3pt]
g_4=-4x_1^6-x_2^2-4x_3^6-x_4^2-4x_5^6-x_6^2+1.25;\quad
&g_5=-2x_1^2-3x_2^6-2x_3^2-3x_4^6-2x_5^2-3x_6^6+1.1;
\quad x\geq 0.
\end{array}
\end{eqnarray*}
Example {\tt C8\_2} (8 variables, degree 2): 
\begin{eqnarray*}
\begin{array}{lll}
f= x_1^2+x_2^2+x_3^2+x_4^2+x_5^2+x_6^2+x_7^2+x_8^2+2x_1x_2-x_1-x_2;\quad
&g_1=-x_1^2-2x_2^2-x_3^2-2x_4^2-x_5^2-2x_6^2-x_7^2-2x_8^2+1;\\[3pt]
g_2=-2x_1^2-x_2^2-2x_3^2-x_4^2-2x_5^2-x_6^2-2x_7^2-x_8^2+1;\quad
&g_3=-x_1^2-4x_2^2-x_3^2-4x_4^2-x_5^2-4x_6^2-x_7^2-4x_8^2+1.25;\\[3pt]
g_4=-4x_1^2-x_2^2-4x_3^2-x_4^2-4x_5^2-x_6^2-4x_7^2-x_8^2+1.25;\quad
&g_5=-2x_1^2-3x_2^2-2x_3^2-3x_4^2-2x_5^2-3x_6^2-2x_7^2-3x_8^2+1.1;
\\[3pt]
x\geq 0.
\end{array}
\end{eqnarray*}
Example {\tt C8\_4} (8 variables, degree 4): 
\begin{eqnarray*}
\begin{array}{lll}
f= x_1^4+x_2^4+x_3^4+x_4^4+x_5^4+x_6^4+x_7^4+x_8^4+3x_1^2x_2^2-x_1-x_2;\quad
&g_1=-x_1^4-2x_2^4-x_3^4-2x_4^4-x_5^2-2x_6^4-x_7^4-2x_8^4+1;\\[3pt]
g_2=-2x_1^4-x_2^4-2x_3^4-x_4^4-2x_5^2-x_6^4-2x_7^4-x_8^4+1;\quad
&g_3=-x_1^4-4x_2^4-x_3^4-4x_4^4-x_5^4-4x_6^4-x_7^4-4x_8^4+1.25;\\[3pt]
g_4=-4x_1^4-x_2^4-4x_3^4-x_4^4-4x_5^4-x_6^4-4x_7^4-x_8^4+1.25;\quad
&g_5=-2x_1^4-3x_2^2-2x_3^4-3x_4^2-2x_5^4-3x_6^2-2x_7^4-3x_8^2+1.1;
\\[3pt] x\geq 0.
\end{array}
\end{eqnarray*}
Example {\tt C10\_2} (10 variables, degree 2): 
\begin{eqnarray*}
\begin{array}{lll}
f&=& x_1^2+x_2^2+x_3^2+x_4^2+x_5^2+x_6^2+x_7^2+x_8^2+x_9^2+x_{10}^2
+2x_1x_2-x_1-x_2;
\\[3pt]
g_1&=&-x_1^2-2x_2^2-x_3^2-2x_4^2-x_5^2-2x_6^2-x_7^2-2x_8^2-x_9^2-2x_{10}^2+1;
\\[3pt]
g_2&=&-2x_1^2-x_2^2-2x_3^2-x_4^2-2x_5^2-x_6^2-2x_7^2-x_8^2-2x_9^2-x_{10}^2+1;
\\[3pt]
g_3&=&-x_1^2-4x_2^2-x_3^2-4x_4^2-x_5^2-4x_6^2-x_7^2-4x_8^2-x_9^2-4x_{10}^2+1.25;\\[3pt]
g_4&=&-4x_1^2-x_2^2-4x_3^2-x_4^2-4x_5^2-x_6^2-4x_7^2-x_8^2-4x_9^2-x_{10}^2+1.25;\\[3pt]
g_5&=&-2x_1^2-3x_2^2-2x_3^2-3x_4^2-2x_5^2-3x_6^2-2x_7^2-3x_8^2-2x_9^2
-3x_{10}^2+1.1;
\\[3pt] x\geq 0.
\end{array}
\end{eqnarray*}
Example {\tt C10\_4} (10 variables, degree 4): 
\begin{eqnarray*}
\begin{array}{lll}
f&=& x_1^4+x_2^4+x_3^4+x_4^4+x_5^4+x_6^4+x_7^4+x_8^4+x_9^4+x_{10}^4
+3x_1^2x_2^2-x_1-x_2;
\\[3pt]
g_1&=&-x_1^4-2x_2^4-x_3^4-2x_4^4-x_5^4-2x_6^4-x_7^4-2x_8^4-x_9^4-2x_{10}^4+1;
\\[3pt]
g_2&=&-2x_1^4-x_2^4-2x_3^4-x_4^4-2x_5^4-x_6^4-2x_7^4-x_8^4-2x_9^4-x_{10}^4+1;
\\[3pt]
g_3&=&-x_1^4-4x_2^4-x_3^4-4x_4^4-x_5^4-4x_6^4-x_7^4-4x_8^4-x_9^4-4x_{10}^4+1.25;\\[3pt]
g_4&=&-4x_1^4-x_2^4-4x_3^4-x_4^4-4x_5^4-x_6^4-4x_7^4-x_8^4-4x_9^4-x_{10}^4+1.25;\\[3pt]
g_5&=&-2x_1^4-3x_2^2-2x_3^4-3x_4^2-2x_5^4-3x_6^2-2x_7^4-3x_8^2-2x_9^4
-3x_{10}^2+1.1;
\\[3pt]
x\geq 0.
\end{array}
\end{eqnarray*}
Example {\tt C20\_2} (20 variables, degree 2): 
\begin{eqnarray*}
\begin{array}{lll}
f&=&x_1^2+x_2^2+x_3^2+x_4^2+x_5^2+x_6^2+x_7^2+x_8^2+x_9^2+x_{10}^2
+ 2x_1x_2-x_1-x_2
\\
&&+x_{11}^2+x_{12}^2+x_{13}^2+x_{14}^2+x_{15}^2+x_{16}^2+x_{17}^2+x_{18}^2+x_{19}^2+x_{20}^2;\\[3pt]
g_1&=&-x_1^2-2x_2^2-x_3^2-2x_4^2-x_5^2-2x_6^2-x_7^2-2x_8^2-x_9^2-2x_{10}^2
\\
&&-x_{11}^2-2x_{12}^2-x_{13}^2-2x_{14}^2-x_{15}^2-2x_{16}^2-x_{17}^2-2x_{18}^2-x_{19}^2-2x_{20}^2+1;
\\[3pt]
g_2&=&-2x_1^2-x_2^2-2x_3^2-x_4^2-2x_5^2-x_6^2-2x_7^2-x_8^2-2x_9^2-x_{10}^2
\\
&&
-2x_{11}^2-x_{12}^2-2x_{13}^2-x_{14}^2-2x_{15}^2-x_{16}^2-2x_{17}^2-x_{18}^2-2x_{19}^2-x_{20}^2+1;
\\[3pt]
g_3&=&-x_1^2-4x_2^2-x_3^2-4x_4^2-x_5^2-4x_6^2-x_7^2-4x_8^2-x_9^2-4x_{10}^2
\\
&&-x_{11}^2-4x_{12}^2-x_{13}^2-4x_{14}^2-x_{15}^2-4x_{16}^2-x_{17}^2-4x_{18}^2-x_{19}^2-4x_{20}^2+1.25;
\\[3pt]
g_4&=&-4x_1^2-x_2^2-4x_3^2-x_4^2-4x_5^2-x_6^2-4x_7^2-x_8^2-4x_9^2-x_{10}^2
\\
&&-4x_{11}^2-x_{12}^2-4x_{13}^2-x_{14}^2-4x_{15}^2-x_{16}^2-4x_{17}^2-x_{18}^2-4x_{19}^2-x_{20}^2+1.25;
\\[3pt]
g_5&=&-2x_1^2-3x_2^2-2x_3^2-3x_4^2-2x_5^2-3x_6^2-2x_7^2-3x_8^2-2x_9^2-3x_{10}^2
\\
&&-2x_{11}^2-3x_{12}^2-2x_{13}^2-3x_{14}^2-2x_{15}^2-3x_{16}^2-2x_{17}^2
-3x_{18}^2-2x_{19}^2-3x_{20}^2+1.1;
\\ 
x\geq 0.
\end{array}
\end{eqnarray*}
\end{scriptsize}


\end{document}